\documentclass[11pt,reqno]{amsart}
\usepackage{a4wide}
\usepackage{amsmath,amsfonts,amssymb,bbm}
\usepackage{verbatim}
\usepackage[shortlabels]{enumitem}
\usepackage[T1]{fontenc}
\usepackage[unicode,hypertexnames=false,colorlinks=true,linkcolor=blue,citecolor=blue]{hyperref}
\usepackage{esint}
\usepackage{hyperref}
\usepackage[capitalise, nameinlink, noabbrev]{cleveref} 
\usepackage{mathrsfs}
\usepackage{float}
\usepackage{cancel}
\usepackage[dvipsnames]{xcolor}
\hypersetup{colorlinks,breaklinks,
	linkcolor=[rgb]{0,0,0},
	citecolor=[rgb]{0,0,0},
	urlcolor=[rgb]{0,0,0}}

\usepackage{xfrac,xcolor,color,graphicx}
\newcommand{\be}{\begin{equation}}
	\newcommand{\ee}{\end{equation}}
\theoremstyle{plain} 

\newtheorem{theorem}{Theorem}[section]
\newtheorem{proposition}[theorem]{Proposition}

\newtheorem{lemma}[theorem]{Lemma}

\newcommand{\R}{\mathbb{R}}

\newcommand{\N}{\mathbb{N}}

\newcommand{\HH}{\mathcal{H}}

\newcommand{\bea}{\begin{equation*}\begin{aligned}}
		\newcommand{\eea}{\end{aligned}\end{equation*}}

\crefname{subsection}{subsection}{subsections}

\numberwithin{equation}{section}

\numberwithin{equation}{section}

\title{A minimizing cone from Cartan polynomials in the Alt-Caffarelli problem}

\author[M. Carducci]{Matteo Carducci}\thanks{}
\address {Matteo Carducci \newline \indent
	Classe di Scienze, Scuola Normale Superiore \newline \indent
	Piazza dei Cavalieri 7, 56126 Pisa - ITALY}
\email{\href{mailto:matteo.carducci@sns.it}{matteo.carducci@sns.it}}
\author[B. Velichkov]{Bozhidar Velichkov}\thanks{}
	\address {Bozhidar Velichkov \newline \indent
		Dipartimento di Matematica, Universit\`a di Pisa \newline \indent
		Largo B. Pontecorvo 5, 56127 Pisa - ITALY}
	\email{\href{mailto:bozhidar.velichkov@unipi.it}{bozhidar.velichkov@unipi.it}}

\begin{document}
	
	\subjclass[2010] {
	}
	
	\keywords{Free boundary, Alt-Caffarelli, isoparametric hypersurfaces, Cartan polynomials}
	\subjclass{35R35}
    \begin{abstract}
        In this paper 
        we study a class of 1-homogeneous solutions of the one-phase problem coming from the Cartan isoparametric cubic polynomials, which exist only in dimensions $d=5,8,14,26$.
        Our main result shows that they are unstable in dimensions $d=5,8,14$, and minimizing in dimension $d=26$. This provides the first example of a singular minimizing cone in the Alt-Caffarelli problem that is neither axially symmetric nor of Lawson-type.
	\end{abstract}
    
    \maketitle
	\tableofcontents

	\maketitle
\section{Introduction}
Among variational free boundary problems, the Alt-Caffarelli problem provides one of the basic models for the regularity theory of unknown interfaces.
It is closely related to several geometric and variational problems, including minimal surface theory. The associated energy functional was introduced by Alt and Caffarelli in \cite{AltCaffarelli:OnePhaseFreeBd}, and is defined for any domain $D\subset\R^d$ and any function $u\in H^1(D)$ as 
\be\label{eq:functional}J(u,D):=\int_{D}|\nabla u|^2\,dx+|\{u>0\}\cap D|.\ee 
Given non-negative boundary data, one minimizes $J$ among functions with the prescribed boundary values. More generally, one may consider non-negative critical points of $J$, which formally satisfy
\be\label{eq:onephase}
\Delta u=0\quad\text{in }\Omega_u\cap D,
\qquad
|\nabla u|=1\quad\text{on }\partial\Omega_u\cap D,
\ee where $\Omega_u:=\{u>0\}$.
The aim of the problem is to study the regularity and the structure of the unknown interface $\partial\Omega_u\cap D$, which is called the \emph{free boundary}. 

Regularity theory for the Alt-Caffarelli problem has been extensively studied over the past decades; 
see, among others, \cite{AltCaffarelli:OnePhaseFreeBd,Caffarelli1,Caffarelli2,Caffarelli3,weiss98,CaffarelliJerisonKenig04:NoSingularCones3D,DeSilvaJerison09:SingularConesIn7D,JerisonSavin15:NoSingularCones4D,esv,kriventsov-weiss,ChanFernandezFigalliSerra} and the monographs \cite{CaffarelliSalsa:GeomApproachToFreeBoundary,Velichkov:RegularityOnePhaseFreeBd}. 
However, a complete theory for the structure of these interfaces remains far from reach.

A fundamental tool in the regularity theory of the one-phase problem \eqref{eq:onephase} is the Weiss' monotonicity formula \cite{weiss98}. Given a Lipschitz stationary solution $u$ of \eqref{eq:onephase} (resp.~a minimizer of \eqref{eq:functional}), rescaling around a free boundary point $x_0\in \partial\Omega_u\cap D$, one obtains $1$-homogeneous global solutions of \eqref{eq:onephase} (resp.~$1$-homogeneous minimizers), that we call blow-up limits. 
If a blow-up is, up to rotations, the half-space solution $(x_d)_+$, then the point is called regular, and the free boundary is locally smooth in a neighborhood of $x_0$ (see \cite{AltCaffarelli:OnePhaseFreeBd,DeSilva:FreeBdRegularityOnePhase,KinderlehrerNirenberg1977:AnalyticFreeBd,kriventsov-weiss}). 

Singular points are instead modeled by non-flat $1$-homogeneous global solutions, which we refer to as \emph{one-phase cones}. If a one-phase cone is a minimizer of \eqref{eq:functional} in any bounded domain $D\subset\R^d$, we call it a \emph{minimizing cone}. One-phase cones are called \emph{stable cones} if the second variation with respect to domain
deformations with compact support away from the origin is non-negative (see, for instance, \cite{JerisonSavin15:NoSingularCones4D}).

For singular points modeled by cones with an isolated singularity, the asymptotic description of the free boundary is well understood \cite{esv,ERZ25,carduccivelichkovuniqueness}. In this case, the free boundary of $u$ is a regular perturbation of the free boundary of the blow-up, in a neighborhood of the singular point. 
The construction and classification of singular one-phase and minimizing cones with isolated singularities are therefore central problems in the regularity theory of the Alt-Caffarelli problem. 

Classical examples of one-phase cones with isolated singularity are the axially symmetric cones obtained in \cite{CaffarelliJerisonKenig04:NoSingularCones3D}. 
The corresponding cone in dimension $d=7$ was proved to be minimizing in \cite{DeSilvaJerison09:SingularConesIn7D}.
One-phase cones with bi-orthogonal symmetry were introduced in \cite{Hong}. They are the analogues of the Lawson-type cones for minimal surfaces \cite{lawson}, and their stability and minimality were studied in \cite{ftw1,ftw2}. Other families of homogeneous solutions were subsequently constructed using different symmetry reductions \cite{wang-hong,wang-cartan,lavoroa3} (see also \cite{ftz}). 
\subsection{Main results}
In this paper, we study the stability and minimality of a family of one-phase cones arising from the Cartan isoparametric cubic polynomials. We refer to them as the \emph{one-phase Cartan cones}. These solutions were introduced in \cite{wang-hong,wang-cartan} as singular critical points of the Alt-Caffarelli functional \eqref{eq:functional}, and subsequently studied in \cite{ftz}. They are the one-phase analogues of the minimal surface cones associated with Cartan's polynomials; see, for instance, \cite{lawson,solomonI,solomonII,lawlor,WangSterling94,cartan-minimal-surfaces}.  

We recall that a Cartan cubic polynomial in dimension $d\ge3$ is a $3$-homogeneous polynomial $P_d:\R^d\to\R$ satisfying
\begin{equation}\label{e:d-Cartan-harmonicity}  
\Delta P_d(x)=0\quad\text{and}\quad|\nabla P_d(x)|^2=9|x|^4\quad\text{for every } x\in\R^d.
\end{equation}
Cartan's classification \cite{cartan} (see also \cite{tkachev,survey-cartan}) shows that such polynomials exist only in dimensions $d=5,8,14,26$, and have the following form.  
Given a division algebra $\mathbb{F}\in\{\R,\mathbb{C},\mathbb{H},\mathbb{O}\}$, set $m:=\dim_{\R}\mathbb{F}\in\{1,2,4,8\}$, and write $x=(y,z,w,x_{d-1},x_d)\in\mathbb{F}^3\times\R^2$. Then, the Cartan cubic polynomial in $\R^d$, with $d=3m+2\in\{5,8,14,26\}$, is defined as \be\label{def:P} P_d(x):=x_d^3-3x_{d-1}^2x_{d}+\frac32x_d(|y|^2+|z|^2-2|w|^2)+\frac{3\sqrt3}2x_{d-1}(|y|^2-|z|^2)+3\sqrt{3}\text{Re}((yz)\overline w),
\ee 
where we denote by $\overline w$ the conjugate of $w$ in $\mathbb{F}$, and by $|\cdot |$ the norm corresponding to the scalar product $\langle y,z\rangle:=\text{Re}(y\overline z)$.
For such dimensions, we consider the $0$-homogeneous function $\widetilde P_d:\R^d\setminus\{0\}\to \R$ defined as \be\label{e:Cartan-widetilde-P}\widetilde P_d(x):=|x|^{-3}P_d(x).\ee
Then, the corresponding one-phase Cartan cone is the 1-homogeneous function 
$U_d:\R^d\to\R,$
defined as
\be\label{eq:general-cone-cartan}U_d(x):=\begin{cases}
		c_d^{-1/2}|x|f_d(\widetilde P_d(x))&\quad\text{if } |\widetilde P_d(x)|\le t_d,\medskip\\
		0&\quad\text{otherwise},
		\end{cases}\qquad \text{for }d=5,8,14,26,\ee
where $f_d:(-1,1)\to\R$ is the even hypergeometric function (see \eqref{def:hyper})
\be\label{def:f}f_d(t):={}_2F_1\left(\frac{d-1}6,-\frac16;\frac12;t^2\right),\ee and $t_d\in(0,1)$ is the unique positive zero of $f_d$ in $(0,1)$, by \cref{lemma:sol-of-ode}. 
The choice of $f_d$ gives that 
$$\Delta U_d=0\quad\text{in}\quad \Omega_{U_d}=\{|\widetilde P_d(x)|<t_d\},$$
while $c_d:=9(1-t_d^2)f_d'(t_d)^2>0$ is a normalization constant that ensures the validity of the free boundary condition 
$$|\nabla U_d|=1\quad\text{on}\quad \partial\Omega_{U_d}\setminus\{0\}.$$
For details about the construction of $U_d$ we refer to \cite[Theorem 1.1]{wang-cartan} (see also \cite{wang-hong,ftz} or \cref{p:existence}). 

Our main result characterizes the stability and minimality of the one-phase Cartan cones. 
\begin{theorem}\label{thm:main}
    Let $U_d$ be the one-phase Cartan cone in $\R^d$, defined in \eqref{eq:general-cone-cartan} for $d=5,8,14,26$. Then, $U_5$, $U_8$ and $U_{14}$ are unstable cones, while $U_{26}$ is a minimizing cone.
\end{theorem}
In particular, \cref{thm:main} provides a new example of a singular minimizing cone for the Alt-Caffarelli problem. 
Unlike the previously known minimizing cones, which are axially symmetric or of Lawson-type, the cone $U_{26}$ arises from the exceptional structure associated with the Cartan isoparametric cubic polynomial in $\R^{26}$. More precisely, given a Cartan cubic polynomial $P_d$, the level sets $\Sigma_t:=\{\widetilde P_d=t\}\cap \mathbb{S}^{d-1}$, $t\in(-1,1)$, are a parallel family of isoparametric hypersurfaces on the sphere, and have three distinct constant principal curvatures. 
Recall that a hypersurface in $\mathbb S^{d-1}$ is called \emph{isoparametric} if its principal curvatures, and so the mean curvature, are constant. 
Among the hypersurfaces in this family, $\Sigma_0$ is minimal in the sphere
\cite{lawson} (see also \eqref{e:Cartan-zero-homogeneous-curvature}).
The cone over $\Sigma_0$ is unstable in dimensions $d=5,8$, whereas it is
area-minimizing in dimensions $d=14,26$ \cite[Theorem~4]{lawson}.
In contrast, \cref{thm:main} shows that the associated one-phase Cartan cone
$U_{14}$ is unstable. 

The geometry of the Cartan cones $U_d$ admits a particularly simple  description in terms of this isoparametric family. The free boundary of $U_d$ is the union of the cones over $\Sigma_{t_d}$ and $\Sigma_{-t_d}$, and the spherical positivity set $\Omega_{U_d}\cap\mathbb{S}^{d-1}$ is the region between these two hypersurfaces, and lies on the two sides of the minimal Cartan hypersurface $\Sigma_0$. 
This provides a direct connection between the Cartan one-phase cones $U_d$ and the geometry of minimal and constant mean curvature hypersurfaces on the sphere. 
\subsection{Strategy of the proof.}
The proof of \cref{thm:main} combines ideas from \cite{CaffarelliJerisonKenig04:NoSingularCones3D,DeSilvaJerison09:SingularConesIn7D}.
First, we study the stability of $U_d$, for $d=5,8,14,26$, by establishing an instability criterion for Cartan cones (see \cref{lemma:criterion}). This is the analogue of the criterion obtained by Caffarelli, Jerison and Kenig \cite{CaffarelliJerisonKenig04:NoSingularCones3D} for axially symmetric cones. 
    By evaluating this criterion numerically, we prove that $U_d$ is unstable for $d=5,8,14$. On the other hand, $U_{26}$ does not satisfy the criterion, suggesting that it may instead be stable.

    Rather than proving the stability of $U_{26}$, we establish directly the stronger conclusion that it is minimizing. 
    To this end, we follow the barrier argument introduced by De Silva and Jerison \cite{DeSilvaJerison09:SingularConesIn7D}. The main idea is to construct a strict subsolution and a strict supersolution living below and above $U_{26}$ respectively, and converging to it under rescaling. Their rescalings provide barriers from which minimality follows (see \cref{lemma:subsol-and-supersol}). 
    While the construction of the subsolution closely follows \cite{DeSilvaJerison09:SingularConesIn7D}, the construction of the supersolution is the main new technical ingredient of the proof. We outline the argument below and refer to \cref{subsec:super} for the complete construction.

\subsubsection{Construction of the supersolution}

The construction of the supersolution is broadly inspired by the one introduced in \cite{DeSilvaJerison09:SingularConesIn7D}, but the structure associated with the Cartan polynomial requires an additional transition region.
We begin by introducing coordinates adapted to the level sets of $\widetilde P$. We set
$\rho:=|x|$, $t:=\widetilde P(x),$ $\theta:=\frac13\arcsin t,$
so that $t=\sin(3\theta)$, and then define
$$s:=\rho\sin\theta,\qquad r:=\rho\cos\theta.$$
Since $|t|\le 1$ (see \cref{lemmaminmax}), we have that $|\theta|\leq\pi/6$. Therefore, the image of the entire space through this change of coordinates is the closed sector
$$\mathcal D:=\left\{(s,r): |s|\leq \frac{r}{\sqrt3}\right\}.$$
The boundary $\partial\mathcal D$ consists of the two lines $s=\pm r/\sqrt3$, which correspond to the level sets
$\{\widetilde P=\pm1\}$. Since these level sets are not hypersurfaces in the original variables (see \cref{lemma3.2}), the change of coordinates degenerates along the lines $s=\pm r/\sqrt3$.

This degeneracy creates a difficulty in the construction. In the new coordinates, the transformed Laplacian contains singular terms on $\partial\mathcal D$.
This motivates our choice to impose the Neumann condition
\be\label{eq:compat:cond}\partial_r w\mp\sqrt{3}\partial_s w=0\quad \text{on }s=\pm r/\sqrt{3}\ee when the positivity set of $w$ reaches $\partial\mathcal{D}$.

We now describe the structure of the construction. The supersolution must be positive in a full neighborhood of the origin, and hence its positivity set must lie entirely in $\mathcal D$, reaching
the lines $s=\pm r/\sqrt3$ in this neighborhood. We follow a construction similar to that in \cite{DeSilvaJerison09:SingularConesIn7D}, until the free boundary of the inner De Silva-Jerison ansatz meets $\partial\mathcal D$ orthogonally  at some $r=r_\ast$. From that point, we introduce an additional profile around the origin, which is strictly positive in $\mathcal D$, for
$r\in(0,r_\ast)$, and satisfies \eqref{eq:compat:cond}.
More precisely, we construct the supersolution in the following form
$$W(x):=\begin{cases}
        W_1(x), & r\in[1,+\infty),\\
        W_2(x), & r\in[r_\ast,1],\\
        W_3(x), & r\in[0,r_\ast].
    \end{cases}$$
The first two pieces play the same roles as the two profiles in
\cite{DeSilvaJerison09:SingularConesIn7D}. The outer piece $W_1$ is obtained by perturbing $U$ through a homogeneous
harmonic function and converges to $U$ under rescaling. The profile
$W_2$ is glued continuously to $W_1$ at $r=1$, allowing a jump of the
normal derivative with the sign required for superharmonicity. Its free
boundary lies strictly inside $\mathcal D$ for $r\in(r_\ast,1)$ and reaches
$\partial\mathcal D$ orthogonally at $r=r_\ast$.

The new inner piece $W_3$ is glued with $W_2$ in a $C^1$ way at $r=r_\ast$. It is positive throughout $\mathcal D$ for
$r\in(0,r_\ast)$ and satisfies \eqref{eq:compat:cond}. In particular,
the positivity set of $W$ contains a neighborhood of the origin. 
The three profiles are subject to several competing requirements, arising from the supersolution inequality, the free boundary condition, the structure of their positivity sets, and the gluing at $r=1$ and $r=r_\ast$. After reducing these requirements to a finite collection of explicit inequalities, we choose suitable parameters and verify the resulting conditions numerically.

\subsection{Structure of the paper} 
The paper is organized as follows. In \cref{sec2}, we recall the construction of the one-phase Cartan cones $U_d$, proving that they are homogeneous solutions. In \cref{sec3}, we study the stability and instability of $U_d$, proving in particular that $U_d$ is unstable for $d=5,8,14$. In \cref{sec4}, we prove that $U_{26}$ is minimizing, by constructing a strict subsolution and a strict supersolution living below and above $U_{26}$ respectively.

\subsection{Use of AI}
No AI tools were used to produce mathematical content; the strategy and the proofs were developed and written exclusively by the authors. AI tools were used only for language editing and a final review of the paper. The numerical computations were done with Mathematica. 
\subsection{Acknowledgements}
The authors are supported by the European Research Council (ERC) via the project ERC FiRM - {\em Fine structure and regularity of stationary and moving free boundaries} (grant agreement No. 101230705). 

\section{Construction of one-phase Cartan cones}\label{sec2}
In this section we recall the main steps in the construction of the one-phase Cartan cones defined in \eqref{eq:general-cone-cartan}. Most of these results were established in \cite{wang-hong,wang-cartan}. The main result of this section is the following proposition.
\begin{proposition}\label{p:existence}
	Let $P_d:\R^d\to\R$ be a $3$-homogeneous polynomial satisfying \eqref{e:d-Cartan-harmonicity}, and let $\widetilde P_d:\R^d\setminus\{0\}\to\R$ be given by \eqref{e:Cartan-widetilde-P}. Let $f_d$ be the function defined in \eqref{def:f}, let $t_d\in(0,1)$ be the unique zero of $f_d$ in $(0,1)$, and let $c_d:=9(1-t_d^2)f_d'(t_d)^2$. Then, the function $U_d$ defined in \eqref{eq:general-cone-cartan} is a global $1$-homogeneous solution of the one-phase problem. 
\end{proposition}
\subsection{The zero-homogeneous part of a cubic Cartan polynomial}
We first recall the following properties for the zero-homogeneous function $\widetilde P$ defined in \eqref{e:Cartan-widetilde-P}.

\begin{lemma}
Let $P:\R^d\to\R$ be a $3$-homogeneous polynomial satisfying \eqref{e:d-Cartan-harmonicity}, and let $\widetilde P:\R^d\setminus\{0\}\to\R$ be given by \eqref{e:Cartan-widetilde-P}. 
Then, 
	\begin{equation}\label{e:Cartan-zero-homogeneous-gradient-norm}
	|\nabla \widetilde P|^2=
	9|x|^{-2}\big(1-\widetilde P^2(x)\big)\quad\text{and}\quad \Delta \widetilde P(x)
	=-3(d+1)|x|^{-2}\widetilde P(x).
\end{equation}	
\end{lemma}	
\begin{proof}
    The result follows by an explicit computation. First we observe that
\begin{equation}\label{e:Cartan-zero-homogeneous-gradient-full}
	\nabla \widetilde P(x)=\frac{1}{|x|^3}\Big(-3\frac{x}{|x|^2}P(x)+\nabla P(x)\Big).
\end{equation}	
Then, we compute
\begin{align*}
|\nabla \widetilde P|^2&=\frac{1}{|x|^6}\Big|-3\frac{x}{|x|^2}P(x)+\nabla P(x)\Big|^2=\frac{1}{|x|^6}\Big(\frac{9}{|x|^2}P^2+|\nabla P|^2-6P\frac{x\cdot \nabla P(x)}{|x|^2}\Big)\\
&=\frac{1}{|x|^6}\Big(\frac{9}{|x|^2}P^2+9|x|^4-6P\frac{3P}{|x|^2}\Big)=\frac{9}{|x|^6}\Big(\frac{1}{|x|^2}P^2+|x|^4-2P^2\frac{1}{|x|^2}\Big)\\
&=\frac{9}{|x|^2}\Big(1-\widetilde P^2(x)\Big).
\end{align*}
For the Laplacian of $\widetilde P$, we have
\begin{align*}
\Delta \widetilde P&=|x|^{-3}\Delta P+2\nabla (|x|^{-3})\cdot\nabla P+P\Delta(|x|^{-3})=2\nabla (|x|^{-3})\cdot\nabla P+P\Delta(|x|^{-3})\\&=-6\frac{x\cdot\nabla P}{|x|^5}-3(d-5)|x|^{-5}P(x)=-6\frac{3P(x)}{|x|^5}-3(d-5)|x|^{-5}P(x)\\&=-3(d+1)\frac{P(x)}{|x|^5},
\end{align*}
where we used that $\Delta(|x|^{-3})=-3(d-5)|x|^{-5}.$ This concludes the proof.
\end{proof}

\begin{lemma}\label{lemmaminmax}
Let $P:\R^d\to\R$ be a $3$-homogeneous polynomial satisfying \eqref{e:d-Cartan-harmonicity}, and let $\widetilde P:\R^d\setminus\{0\}\to\R$ be given by \eqref{e:Cartan-widetilde-P}. 
	Then, 
	\begin{equation*}
	\max_{x\in\R^d\setminus\{0\}}\widetilde P(x)=1\qquad\text{and}\qquad 	\min_{x\in\R^d\setminus\{0\}}\widetilde P(x)=-1.
	\end{equation*}	
\end{lemma}	
\begin{proof}
Since $\widetilde P$ is zero-homogeneous, the minimum and the maximum are clearly achieved. Thanks to \eqref{e:Cartan-zero-homogeneous-gradient-full}, the critical points $x$ for $\widetilde P$ are solutions to
$$\nabla P(x)=3P(x)\frac{x}{|x|^2}.$$
In particular, taking the norm on both sides and using \eqref{e:d-Cartan-harmonicity}, we obtain 
$|P(x)|=|x|^3.$
Thus, at all critical points $x$ for $\widetilde P$ we have $|\widetilde P(x)|=1$. Since $\widetilde P$ is not constant, we conclude.
\end{proof}

\subsection{Laplacian and gradient of $|x|^\alpha\phi(\widetilde P(x))$}
In order to construct the one-phase Cartan cones, we prove the following general lemma.

\begin{lemma}\label{lemma:generallemma}
Let $P:\R^d\to\R$ be a $3$-homogeneous polynomial satisfying \eqref{e:d-Cartan-harmonicity}, and let $\widetilde P:\R^d\setminus\{0\}\to\R$ be given by \eqref{e:Cartan-widetilde-P}. Given $\alpha\in\R$ and a $C^2$ function $\phi:(-1,1)\to\R$, we consider the function  
$u:\{|\widetilde P|\le 1\}\setminus\{0\}\to\R$ defined as $u(x):=|x|^{\alpha}\phi(\widetilde P(x)).$
Then, 
\begin{equation}\label{e:general-alpha-homogeneous-laplacian}
	\Delta u(x)=|x|^{\alpha-2}\Big(\alpha(\alpha+d-2)\phi(\widetilde P)-3(d+1)\phi'(\widetilde P)\widetilde P+9\phi''(\widetilde P)\big(1-\widetilde P^2\big)\Big).
\end{equation}	
Moreover, 
\begin{equation}\label{e:general-alpha-homogeneous-gradient-norm}
	|\nabla u(x)|^2=|x|^{2\alpha-2}\Big(\alpha^2\phi^2(\widetilde P)+9\big(1-\widetilde P^2\big)(\phi'(\widetilde P))^2	\Big),	 	
\end{equation}	
and
\begin{equation}\label{e:general-alpha-homogeneous-gradient-normal}
	\frac{\nabla \widetilde P}{|\nabla \widetilde P|}\cdot \nabla u(x)=3|x|^{\alpha-1}\phi'(\widetilde P)\big(1-\widetilde P^2\big)^{1/2}.
\end{equation}		
\end{lemma}	
\begin{proof}
Using the identities in \eqref{e:Cartan-zero-homogeneous-gradient-norm},
we compute 
\begin{align*}
	\Delta \widetilde u&=\text{\rm div}\Big(\phi'(\widetilde P)\nabla\widetilde P\Big)=\phi'(\widetilde P)\Delta \widetilde P+\phi''(\widetilde P)|\nabla\widetilde P|^2\\
		&=\phi'(\widetilde P)\Big(-3(d+1)|x|^{-2}\widetilde P\Big)+\phi''(\widetilde P)	9|x|^{-2}\Big(1-\widetilde P^2\Big)\\
	&=\frac{
	9}{|x|^2}\Big(-\frac{d+1}{3}\phi'(\widetilde P)\widetilde P+\phi''(\widetilde P)\big(1-\widetilde P^2\big)\Big).
\end{align*}	
Thus, using the expression of the Laplacian in polar coordinates, we get  
\begin{align*}
\Delta u&=|x|^{\alpha-2}\Big(\alpha(\alpha+d-2)\widetilde u+|x|^2\Delta \widetilde u\Big)\\
&=|x|^{\alpha-2}\Big(\alpha(\alpha+d-2)\phi(\widetilde P)-3(d+1)\phi'(\widetilde P)\widetilde P+9\phi''(\widetilde P)\big(1-\widetilde P^2\big)\Big),
\end{align*}	 
which concludes the proof of \eqref{e:general-alpha-homogeneous-laplacian}. We next compute 
\begin{align*}
|\nabla u(x)|^2&=\alpha^2|x|^{2\alpha-2}\widetilde u^2(x)+|x|^{2\alpha}|\nabla \widetilde u(x)|^2
=\alpha^2|x|^{2\alpha-2}\phi^2(\widetilde P)+|x|^{2\alpha}(\phi'(\widetilde P))^2|\nabla \widetilde P|^2\\
&=\alpha^2|x|^{2\alpha-2}\phi^2(\widetilde P)+|x|^{2\alpha-2}(\phi'(\widetilde P))^2	9\big(1-\widetilde P^2\big), 	 	
\end{align*}	
which gives \eqref{e:general-alpha-homogeneous-gradient-norm}. Finally, in order to prove \eqref{e:general-alpha-homogeneous-gradient-normal}, we observe that 
\begin{align*}
\nabla u(x)&=\nabla\Big(|x|^\alpha\widetilde u(x)\Big)=\alpha x|x|^{\alpha-2}\widetilde u(x)+|x|^\alpha\nabla\widetilde u(x)=\alpha x|x|^{\alpha-2}\phi(\widetilde P)+|x|^\alpha\phi'(\widetilde P)\nabla\widetilde P.
\end{align*}	
Thus, 
\begin{align*}
	\nabla \widetilde P\cdot \nabla u(x)
	&=|x|^\alpha\nabla \widetilde P\cdot\Big(\alpha x|x|^{-2}\phi(\widetilde P)+\phi'(\widetilde P)\nabla\widetilde P\Big).
\end{align*}	
Now, since $\nabla \widetilde P\cdot x=0$,	
we get 
\begin{align*}
	\nabla \widetilde P\cdot \nabla u(x)
	&=|x|^\alpha\phi'(\widetilde P)|\nabla\widetilde P|^2.
\end{align*}	
Using \eqref{e:Cartan-zero-homogeneous-gradient-norm}, the previous identity gives precisely \eqref{e:general-alpha-homogeneous-gradient-normal}. 
\end{proof}	
\subsection{Construction of $f$} As a consequence of \cref{lemma:generallemma}, for a $1$-homogeneous function of the form $|x| \phi(\widetilde P(x))$, the one-phase problem \eqref{eq:onephase} is reduced to an ODE eigenvalue problem for $\phi$. By solving this eigenvalue problem, we obtain the function $f$ in \eqref{def:f}.
\begin{lemma}\label{lemma:sol-of-ode}
For every $d\in\N$, $d\ge 5$, there exists an even function $f:(-1,1)\to\R$ solving the ODE 
\begin{equation*}
\begin{cases}
    9(1-t^2)f''(t)-3(d+1)tf'(t)+(d-1)f(t)=0\quad\text{in } (-1,1),\\
    f(0)=1,\ f'(0)=0.
\end{cases}
\end{equation*}	
Precisely, $f$ coincides with \eqref{def:f}, namely $$f(t)={}_2F_1\left(\frac{d-1}6,-\frac16;\frac12;t^2\right).$$
    Finally, $f$ has a unique zero in $(0,1)$, which we will denote by $t_d$.
\end{lemma}	
\begin{proof}
The fact that $f$ coincides with \eqref{def:f} follows by \cref{lemma:hypergeometric-appendix}. Moreover, by \eqref{eq:derivative-hyper}, \be\label{eq:deriv}f'(t)=-\frac{(d-1)t}{9}{}_2F_1\left(\frac{d+5}6,\frac56;\frac32;t^2\right).\ee Then, by definition of the hypergeometric functions \eqref{def:hyper}, $f'(t)<0$ for every $t\in(0,1)$. 
On the other hand, $\lim_{t\to1^-}f(t)=-\infty$.
This implies that there exists a unique zero of $f$ in $(0,1)$, concluding the proof.
\end{proof}

\subsection{Construction of the Cartan cones}
Combining the results of this section, we conclude the construction of the one-phase Cartan cones in dimensions $d=5,8,14,26$.
\begin{proof}[Proof of \cref{p:existence}]
    The result follows by combining \eqref{e:general-alpha-homogeneous-laplacian}, \eqref{e:general-alpha-homogeneous-gradient-norm} and \cref{lemma:sol-of-ode}.
\end{proof}

\section{Instability of \texorpdfstring{$U_5$}{U5}, \texorpdfstring{$U_8$}{U8} and \texorpdfstring{$U_{14}$}{U14}}\label{sec3}
In this section we study the stability of the one-phase Cartan cones defined in \eqref{eq:general-cone-cartan}. 
The main result of this section is the following proposition.
\begin{proposition}\label{prop:main1}
    Let $U_d$ be the one-phase Cartan cone in $\R^d$, defined in \eqref{eq:general-cone-cartan} for $d=5,8,14,26$.
    Then $U_5$, $U_8$ and $U_{14}$ are unstable cones for the Alt-Caffarelli functional \eqref{eq:functional}.
\end{proposition}
\subsection{Computation of the mean curvature}
In the following lemma we compute the mean curvature $H$ of the level sets of $\widetilde P$.

\begin{lemma}\label{lemma3.2}
Let $P:\R^d\to\R$ be a $3$-homogeneous polynomial satisfying \eqref{e:d-Cartan-harmonicity}, and let $\widetilde P:\R^d\setminus\{0\}\to\R$ be given by \eqref{e:Cartan-widetilde-P}. 
	Then, for every $t\in(-1,1)$, the level set $\{\widetilde P=t\}$ is a smooth $(d-1)$-manifold in $\R^d\setminus\{0\}$, and its mean curvature at a point $x\in\{\widetilde P=t\}\subset \R^d\setminus\{0\}$  is given by
	\begin{equation}\label{e:Cartan-zero-homogeneous-curvature}
H(x)=\frac{d-2}{|x|}\frac{|t|}{\sqrt{1-t^2}}.
	\end{equation}	
\end{lemma}	
\begin{proof}
Without loss of generality, let $t\in[0,1)$ and consider the level set $\{\widetilde P=t\}$. Since the normal to this level set is given by $\nabla \widetilde P/|\nabla\widetilde P|$, then the mean curvature of $\{\widetilde P=t\}$ is given by the formula:
\begin{align*}
	H&=-\text{\rm div}\left(\frac{\nabla\widetilde P}{|\nabla\widetilde P|}\right)=-\frac{\Delta \widetilde P}{|\nabla \widetilde P|}+\frac{\nabla \widetilde P\cdot\nabla^2\widetilde P[\nabla \widetilde P]}{|\nabla \widetilde P|^3}.
\end{align*}
The expressions for $\Delta\widetilde P$ and $|\nabla\widetilde P|$ are already given by \eqref{e:Cartan-zero-homogeneous-gradient-norm}. Thus, we only need to compute the term $\nabla \widetilde P\cdot\nabla^2\widetilde P[\nabla \widetilde P]$. In order to do so, by differentiating the first identity in \eqref{e:Cartan-zero-homogeneous-gradient-norm} and using \eqref{e:Cartan-zero-homogeneous-gradient-full}, we get 
\begin{align*}
2\nabla^2\widetilde P[\nabla\widetilde P]&=\nabla \Big(\frac{9}{|x|^2}(1-\widetilde P^2)\Big)=-18(1-\widetilde P^2)\frac{x}{|x|^4}-18\frac{\widetilde P}{|x|^2} \nabla\widetilde P\\
&=-18(1-\widetilde P^2)\frac{x}{|x|^4}-18\frac{\widetilde P}{|x|^2}\frac{1}{|x|^3}\Big(-3\frac{x}{|x|^2}P(x)+\nabla P(x)\Big)\\
&=-18(1-4\widetilde P^2)\frac{x}{|x|^4}-18\frac{\widetilde P}{|x|^5}\nabla P(x).
\end{align*}
Combining this identity with \eqref{e:Cartan-zero-homogeneous-gradient-full}, we obtain 
\begin{align*}
\nabla\widetilde P\cdot	\nabla^2\widetilde P&[\nabla\widetilde P]=\left(-3\frac{x}{|x|^2}\widetilde P(x)+\frac{1}{|x|^3}\nabla P(x)\right)\cdot \bigg(-9(1-4\widetilde P^2)\frac{x}{|x|^4}-9\frac{\widetilde P}{|x|^5}\nabla P(x)\bigg)\\
&=9\bigg(3(1-4\widetilde P^2)\frac{\widetilde P}{|x|^4}+3\widetilde P^2\frac{1}{|x|^7}(x\cdot \nabla P)-(1-4\widetilde P^2)\frac{1}{|x|^7}(x\cdot \nabla P)-\frac{\widetilde P}{|x|^8}|\nabla P|^2\bigg).
\end{align*} Then, using that $P$ is 3-homogeneous and the second identity in \eqref{e:d-Cartan-harmonicity}, we get
\begin{align*}
\nabla\widetilde P\cdot	\nabla^2\widetilde P[\nabla\widetilde P]&=9\bigg(3(1-4\widetilde P^2)\frac{\widetilde P}{|x|^4}+3\widetilde P^2\frac{1}{|x|^7}3P-(1-4\widetilde P^2)\frac{1}{|x|^7}3P-\frac{\widetilde P}{|x|^8}9|x|^4\bigg)\\
&=\frac{9}{|x|^4}\left({3(1-4\widetilde P^2)\widetilde P}+9\widetilde P^3-{3(1-4\widetilde P^2)\widetilde P}-9\widetilde P\right)=-\frac{81}{|x|^4}\widetilde P\left(1-\widetilde P^2\right).
\end{align*}
Thus, for the mean curvature $H$, we compute 
\begingroup\allowdisplaybreaks
\begin{align*}	
H&=-\frac{\Delta \widetilde P}{|\nabla \widetilde P|}+\frac{\nabla \widetilde P\cdot\nabla^2\widetilde P[\nabla \widetilde P]}{|\nabla \widetilde P|^3}\\
&=\frac{1}{3|x|^{-1}(1-\widetilde P^2)^{1/2}}3(d+1)|x|^{-2}\widetilde P-\frac{1}{27|x|^{-3}(1-\widetilde P^2)^{3/2}}\frac{81}{|x|^4}\widetilde P\big(1-\widetilde P^2\big)\\
&=\frac{1}{|x|^{-1}(1-\widetilde P^2)^{1/2}}(d+1)|x|^{-2}\widetilde P-\frac{1}{|x|^{-3}(1-\widetilde P^2)^{1/2}}\frac{3}{|x|^4}\widetilde P\\
&=\frac{d-2}{|x|}\frac{\widetilde P}{(1-\widetilde P^2)^{1/2}},
\end{align*}
\endgroup
which concludes the proof.
\end{proof}

\subsection{A criterion for the instability of the Cartan cones}
In the following lemma we show an instability criterion for the Cartan cones. This is the analogue of the criterion in \cite{CaffarelliJerisonKenig04:NoSingularCones3D} for axially symmetric solutions.
\begin{lemma}\label{lemma:criterion}
Let $P:\R^d\to\R$ be a $3$-homogeneous polynomial satisfying \eqref{e:d-Cartan-harmonicity}, and let $\widetilde P:\R^d\setminus\{0\}\to\R$ be given by \eqref{e:Cartan-widetilde-P}. Let $f$ be the function defined in \eqref{def:f}, and let $t_d\in(0,1)$ be the unique zero of $f$ in $(0,1)$. 
Suppose that there exists a solution $g:[-t_d,t_d]\to\R$ of the problem 
\be\label{eq:odeg}\begin{cases}
    9(1-t^2)g''(t)-3(d+1)tg'(t)-\left(\frac{d-2}{2}\right)^2g(t)=0\quad\text{in } (-t_d,t_d),\\
    g(0)=1,\  g'(0)=0
\end{cases}\ee and that $g$ satisfies the inequality
\be\label{eq:ineq-g}\frac{g'(t_d)}{g(t_d)}<\frac{d-2}{3}\frac{t_d}{1-t_d^2}.\ee
Then, the Cartan cone $U_d$ defined in \eqref{eq:general-cone-cartan} is unstable in $\R^d$.

\end{lemma}	
\begin{proof}
We recall that the stability inequality of the Alt-Caffarelli problem is (see \cite{CaffarelliJerisonKenig04:NoSingularCones3D})
$$Q(\varphi):=\int_{\Omega_U} |\nabla \varphi|^2\,dx- \int_{\partial \Omega_U} H\varphi^2\,d\HH^{d-1}\ge0\quad\text{for every }\varphi\in C^1_c(\R^d\setminus\{0\}),$$ where $H$ is the mean curvature of the free boundary $\partial \Omega_U$ and $U:=U_d$.

We will use as a test function the function 
$w(x):=|x|^{-\frac{d-2}{2}}g(\widetilde P(x)),$	for $|\widetilde P(x)|\le t_d$, in the stability inequality above. Notice that, since $\Omega_U:=\{|\widetilde P(x)|< t_d\}$, it is sufficient to define $w$ for $|\widetilde P(x)|\le t_d$. 
Moreover, since $g$ solves the ODE \eqref{eq:odeg}, then $w$ is harmonic in $\Omega_U \setminus\{0\}$.

Since the integrals involving $w$ in the stability condition diverge near $0$ and at infinity, we need to multiply it by a cutoff function. Precisely, for $R>1$ large to be chosen, we will use $\varphi(x):=\eta_R(|x|)w(x)$, where $\eta_R$ is defined as 
$$\eta_R(\rho):=\begin{cases}
    0 &\quad\text{in }[0,R^{-2}],\\
    \log(R^2\rho)/\log R &\quad\text{in }[R^{-2},R^{-1}],\\
    1 &\quad\text{in }[R^{-1},R],\\
    \log(R^2/\rho)/\log R  &\quad\text{in }[R,R^2],\\
    0 &\quad\text{in }[R^2,+\infty).
\end{cases}$$
Using that $w$ is harmonic in $\Omega_U\setminus\{0\}$, we have $$\text{div}(w\eta_R^2\nabla w)=|\nabla w|^2\eta_R^2+2w\eta_R\nabla w\cdot\nabla \eta_R=|\nabla (\eta_Rw)|^2-w^2|\nabla\eta_R|^2.$$ Then, by the divergence theorem, we get
\be\label{eq:stab-eq}Q(\varphi)=Q(\eta_Rw)=\int_{\partial\Omega_U}\eta_R^2 w(\partial_\nu w-Hw)\,d\HH^{d-1}+\int_{\Omega_U}w^2|\nabla \eta_R|^2\,dx,\ee
where $\nu$ is the outward normal to $\partial\Omega_U$. 

We recall the spherical free boundary $\partial\Omega_U\cap \mathbb{S}^{d-1}=\Sigma_{t_d}\cup \Sigma_{-t_d}$, where $\Sigma_t:=\{\widetilde P=t\}\cap \mathbb{S}^{d-1}$. Since the two contributions of \eqref{eq:stab-eq} on $\Sigma_{t_d}$ and $\Sigma_{-t_d}$ are the same, we only compute the quantities above for $\Sigma_{t_d}$.
 Thanks to \eqref{e:Cartan-zero-homogeneous-curvature}, we know that 
$$H(x)=\frac{d-2}{|x|}\frac{t_d}{\sqrt{1-t_d^2}}\quad\text{on }\Sigma_{t_d}.$$
On the other hand, using \eqref{e:general-alpha-homogeneous-gradient-normal} on the portion of the boundary $\Sigma_{t_d}$, we compute 
$$\partial_\nu w= \frac{\nabla \widetilde P}{|\nabla \widetilde P|}\cdot\nabla w=3|x|^{-\frac{d-2}{2}-1}g'(\widetilde P)\big(1-\widetilde P^2\big)^{1/2}=3|x|^{-d/2}g'(t_d)\sqrt{1-t_d^2}.$$
Therefore, on $\Sigma_{t_d}$ it holds 
\bea\frac{\partial_\nu w}{w}=\frac1{|x|}\frac{3g'(t_d)}{g(t_d)}\sqrt{1-t_d^2},\eea
and so, 
$$\frac{\partial_\nu w}{w}-H=\frac1{|x|}\left(\frac{3g'(t_d)}{g(t_d)}\sqrt{1-t_d^2}-(d-2)\frac{t_d}{\sqrt{1-t_d^2}}\right)=:-\frac{\delta_d}{|x|},$$ where $\delta_d>0$ by \eqref{eq:ineq-g}. We can write the quadratic form $Q(\varphi)$ in terms of $\delta_d$ as:
$$Q(\varphi)=Q(\eta_Rw)=-\delta_d\int_{\partial\Omega_U}\eta_R^2 \frac{w^2}{|x|}\,d\HH^{d-1}+\int_{\Omega_U}w^2|\nabla \eta_R|^2\,dx.$$
For $\rho=|x|$, we have $w^2=\rho^{-(d-2)}g(t)^2$ on $\Omega_U$ and $d\HH^{d-1}_{\partial\Omega_U}=\rho^{d-2}\,d\HH^{d-2}_{\Sigma_{t_d}}\,d\rho$. Therefore, the quadratic form $Q$ can be estimated by $$Q(\varphi)\le -C_1\delta_d\int_0^{+\infty}\frac{|\eta_R(\rho)|^2}{\rho}\,d\rho+C_2\int_0^{+\infty}\rho|\eta_R'(\rho)|^2\,d\rho,$$ for some constants $C_1,C_2>0$ depending on $\|g\|_{L^\infty((0,t_d))}$ and $d$. Using the definition of $\eta_R$, we obtain that $$Q(\varphi)\le -\widetilde C_1\delta_d\log R+\frac{\widetilde C_2}{\log R},$$ and the right-hand side above is negative for $R$ sufficiently large, concluding the proof.
\end{proof}	
\subsection{Instability of $U_d$, for $d=5,8,14$}
Using \cref{lemma:criterion}, we can conclude the proof of instability of the Cartan cones $U_d$, in dimensions $d=5,8,14$.
\begin{proof}[Proof of \cref{prop:main1}]
    We consider the function $g_d:(-1,1)\to\R$ defined as $$g_d(t)={}_2F_1\left(\frac{d-2}{12},\frac{d-2}{12};\frac12;t^2\right).$$ By \cref{lemma:hypergeometric-appendix}, we obtain that $g_d$ solves the ODE \eqref{eq:odeg}. We set 
    $$\Delta_d:=\frac{g_d'(t_d)}{g_d(t_d)}-\frac{d-2}{3}\frac{t_d}{1-t_d^2},$$
    and numerically (we used Mathematica) we compute $$\Delta_5\approx-117.84<0,\qquad \Delta_8\approx-6.31<0,\qquad \Delta_{14}\approx-1.21<0.
    $$ Therefore, we conclude by \cref{lemma:criterion}.
\end{proof}

\section{Minimality of \texorpdfstring{$U_{26}$}{U26}}\label{sec4}
In this section we prove the minimality of $U_{26}$, showing the following proposition.
\begin{proposition}\label{prop:main2}
    Let $U_{26}$ be the one-phase Cartan cone in $\R^{26}$, defined in \eqref{eq:general-cone-cartan}. Then, $U_{26}$ is a minimizing cone for the Alt-Caffarelli functional \eqref{eq:functional}.
\end{proposition}
In order to prove \cref{prop:main2}, we use the following lemma which was essentially proved in \cite{DeSilvaJerison09:SingularConesIn7D} (see also \cite[Lemma 2.4]{ftw1}).

\begin{lemma}\label{lemma:subsol-and-supersol}
    Let $U$ be a $1$-homogeneous global solution of \eqref{eq:onephase}. Assume that:
    \begin{itemize}
        \item[(i)] there exists a strict subsolution $V\le U$ in $\R^d$, satisfying $V_\rho(x):=\rho^{-1}V(\rho x)\to U$ as $\rho\to+\infty$, with $\partial\Omega_V\cap \{0\}=\emptyset$;
        \item[(ii)] there exists a strict supersolution $W\ge U$ in $\R^d$, satisfying $W_\rho(x):=\rho^{-1}W(\rho x)\to U$ as $\rho\to+\infty$, with $B_{r_\ast}\subset \Omega_W$, for some $r_\ast>0$.
    \end{itemize} Then, $U$ is minimizing for the Alt-Caffarelli functional.
\end{lemma}
For the definition of strict subsolutions and supersolutions of the Alt-Caffarelli problem, we refer to \cite{DeSilvaJerison09:SingularConesIn7D}.

\medskip

In the following two subsections, we construct a subsolution and a supersolution satisfying \cref{lemma:subsol-and-supersol}. 
We work in dimension $d:=26$.
For simplicity, we set $P:=P_{26}$, $U:=U_{26}$ and $f:=f_{26}$, as defined in \eqref{def:P}, \eqref{eq:general-cone-cartan} and \eqref{def:f}. We also recall that $t_d\in(0,1)$ is the unique zero of $f$ in $(0,1)$.

\subsection{Construction of the subsolution}\label{subsec:sub}
We construct a subsolution satisfying \cref{lemma:subsol-and-supersol}, for $U:=U_{26}$. We follow the ideas from \cite{DeSilvaJerison09:SingularConesIn7D}. We write \be\label{eq:def-z}U(x)=c_d^{-1/2}Z(x)\quad\text{where}\quad Z(x):=\begin{cases}
    |x|f(\widetilde P(x))&\quad\text{if } |\widetilde P(x)|\le t_d,\medskip\\
		0&\quad\text{otherwise},
\end{cases}\ee and we recall that $c_d=9(1-t_d^2)f'(t_d)^2>0.$
For $d=26$ and $\alpha\in(-(d-2),0)$ to be chosen, we define $$V(x)=c_d^{-1/2}\left(Z(x)-|x|^\alpha g_\alpha(\widetilde P(x))\right)_+,$$
where \be\label{eq.def:galp}g_\alpha(t):={}_2F_1\left(\frac{\alpha+d-2}6,-\frac\alpha6;\frac12;t^2\right).\ee 
Notice that, since $\alpha<1$, we have that $V_\rho(x):=\rho^{-1}V(\rho x)\to U$ as $\rho\to+\infty$. 
Additionally, since $\alpha\in(-(d-2),0)$, we have that  $g_\alpha>0$ in $(-1,1)$, by definition of the hypergeometric function $g_\alpha$ (see \eqref{def:hyper}). Thus, $V\le U$ in $\R^d$ and $V\equiv0$ in a neighborhood of the origin. In particular, $\partial\Omega_V\cap \{0\}=\emptyset$. Then, in order to prove that $V$ satisfies \cref{lemma:subsol-and-supersol}, we only need to show that it is a strict subsolution.

By definition of $g_\alpha$, we have that (see \cref{lemma:hypergeometric-appendix})
\bea\begin{cases}
    9(1-t^2)g_\alpha''(t)-3(d+1)tg_\alpha'(t)+\alpha(\alpha+d-2)g_\alpha(t)=0\quad\text{in } (-t_d,t_d),\\
    g_\alpha(0)=1,\  g_\alpha'(0)=0.
\end{cases}\eea
Then, by \cref{lemma:generallemma}, $V$ is harmonic in $\{V>0\}$. We now compute the gradient of $V$. We use that, for $t=\widetilde P(x)$, \begin{align*}
\nabla (|x|^\alpha \phi(t))&=\alpha x|x|^{\alpha-2}\phi(t)+|x|^\alpha\nabla(\phi(t))=\alpha x|x|^{\alpha-2}\phi(t)+|x|^\alpha\phi'(t)\nabla\widetilde P(x).
\end{align*} 
Therefore,
\begin{align*}c_d^{1/2}\nabla V(x)&=\nabla Z(x)-\nabla(|x|^\alpha g_\alpha(t))\\&=f(t)\frac{x}{|x|}+|x|f'(t)\nabla \widetilde P-\alpha x|x|^{\alpha-2}g_\alpha(t)-|x|^\alpha g_\alpha'(t)\nabla\widetilde P.\end{align*}
Using that $\nabla\widetilde P\cdot x=0$ and $|\nabla \widetilde P|^2=9|x|^{-2}(1-t^2)$, by \eqref{e:Cartan-zero-homogeneous-gradient-norm}, we get
$$c_d|\nabla V(x)|^2=(f(t)-\alpha |x|^{\alpha-1}g_\alpha(t))^2+9(1-t^2)\left(f'(t)-|x|^{\alpha-1}g_\alpha'(t)\right)^2.$$ Since $f(t)=|x|^{\alpha-1}g_\alpha(t)$ on $\partial\Omega_V$, we obtain 
$$|\nabla V(x)|^2=c_d^{-1}G(\alpha,t)\quad\text{on }\partial\Omega_V,$$ where \be\label{eq:Galophat}G(\alpha,t):=(1-\alpha)^2f(t)^2+9(1-t^2)\left(f'(t)-\frac{f(t)}{g_\alpha(t)}g_\alpha'(t)\right)^2.\ee
Since $G(\alpha,t_d)=c_d$, we need to choose $\alpha\in(-(d-2),0)$ such that $$G(\alpha,t)> G(\alpha,t_d)\quad\text{for every }t\in(-t_d,t_d).$$ 
Numerically (we used Mathematica), one can check that $\alpha:=-16$ works, as shown in the figure below.
\begin{figure}[H]
  \centering
  \includegraphics[width=0.45\textwidth]{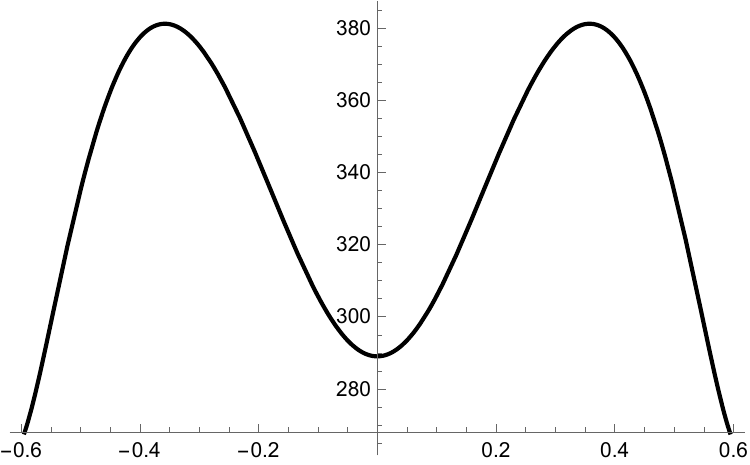}
  \caption{The graph of $G(\alpha,t)$ for $t\in [-t_d,t_d]$.}
\end{figure}

\subsection{Construction of the supersolution}\label{subsec:super} The goal of this subsection is to construct a strict supersolution satisfying \cref{lemma:subsol-and-supersol}, for $U:=U_{26}$. 
Throughout the subsection, we use the following coordinates adapted to the level sets of $\widetilde P$. 
We set
$$\rho:=|x|,\qquad t:=\widetilde P(x),\qquad
    \theta:=\frac13\arcsin t,$$
so that $t=\sin(3\theta)$, and then define
$$s:=\rho\sin\theta,\qquad r:=\rho\cos\theta.$$

The construction of the supersolution is inspired by \cite{DeSilvaJerison09:SingularConesIn7D}, but the structure of the Cartan cubic polynomials introduces an additional difficulty. Indeed, 
since $|t|\le 1$ by \cref{lemmaminmax}, then $|\theta|\le \pi/6$, and so a point in the coordinates $(s,r)$ corresponds to points in $\R^d$ if and only if it belongs to the set $$\mathcal{D}:=\left\{(s,r): |s|\le \frac{r}{\sqrt3}\right\}.$$
Along the lines $s=\pm r/\sqrt{3}$, we have $t=\pm 1$, and the change of coordinates $\theta=\frac13\arcsin t$  becomes singular. 
We construct the supersolution $W$ of the form
\begin{equation}\label{eq:def-W}W(x):=\begin{cases}
    W_1(x)&\quad\text{for $r\in[1,+\infty)$},\\
    W_2(x)&\quad\text{for $r\in[r_\ast,1]$},\\
    W_3(x)&\quad\text{for $r\in[0, r_\ast]$},\\
\end{cases}\end{equation}
for some $r_\ast\in(0,1)$ to be chosen later. 
At $r=1$, we allow a jump of the normal derivative, as in \cite{DeSilvaJerison09:SingularConesIn7D}, while at $r=r_\ast$, we require that the gluing is $C^1$.

\subsubsection{Derivatives in the new coordinates}
Let us start by computing the gradient and the Laplacian in the coordinates $(s,r)$. Recalling the identities in \eqref{e:Cartan-zero-homogeneous-gradient-norm}, we have $$\Delta \rho=\frac{d-1}{\rho},\qquad\Delta\theta=-\frac{d-2}{\rho^2}\tan(3\theta),\qquad|\nabla\theta|^2=\frac1{\rho^{2}},\qquad|\nabla \rho|^2=1,\qquad\nabla\rho\cdot\nabla\theta=0.$$ Moreover, $$ |\nabla s|^2=|\nabla r|^2=1,\qquad\nabla s\cdot\nabla r=0.$$ Additionally, the Laplacians of $s$ and $r$ are given by
$$\Delta s=\frac{d-2}{\rho}(\sin \theta-\cos\theta\tan3\theta)=\frac{d-2}{\rho^2}\left(s-r\frac{2sr^2+(r^2-s^2)s}{(r^2-s^2)r-2s^2r}\right)=-\frac{2(d-2)s}{r^2-3s^2},$$
$$\Delta r=\frac{d-2}{\rho}(\cos \theta+\sin\theta\tan3\theta)=\frac{d-2}{\rho^2}\left(r+s\frac{2sr^2+(r^2-s^2)s}{(r^2-s^2)r-2s^2r}\right)=\frac{(d-2)(r^2-s^2)}{r(r^2-3s^2)}.$$
Therefore, $W(x)=w(s,r)$ is a strict supersolution if  $$\mathcal{L}w:=\partial_{ss}w+\partial_{rr}w-\frac{2(d-2)s}{r^2-3s^2}\partial_sw+\frac{(d-2)(r^2-s^2)}{r(r^2-3s^2)}\partial_rw\le0\quad\text{in }\Omega_W$$ and $$
    (\partial_rw)^2+(\partial_s w)^2<1\quad\text{on }\partial\Omega_W.$$
    Note that, in the interior of $\mathcal{D}$, the operator $\mathcal{L}$ is well-defined.

\subsubsection{Definition of $W$}
We construct the supersolution $W$ as in \eqref{eq:def-W}.
We will choose $W_i(x)=w_i(s,r)$, $i=1,2,3$, for some functions $w_i$. Then, we need to prove the following conditions.
\begin{itemize}
    \item The functions $w_1$, $w_2$ and $w_3$ are superharmonic for $\mathcal{L}$ in their positivity set.
    \item The gradients of $w_1$ and $w_2$ are strictly less than 1 on their free boundaries, while $w_3$ will be chosen positive for $r\in(0,r_\ast)$, and thus it does not have free boundary.
    \item The derivative in the $r$ direction of $w_2-w_1$ must be non-negative. This ensures that the gluing of $W_1$ and $W_2$ is superharmonic in its positivity set. On the other hand, we will choose $w_2$ and $w_3$ in such a way that the gluing of $W_2$ and $W_3$ is $C^1$, implying that it is also superharmonic in its positivity set.
\end{itemize}

    We recall the function $Z$ defined in \eqref{eq:def-z}.
    In the coordinates $(\rho,t)$, we have $Z(x)=\rho f(t)$. 
    The positivity set of $U$ is given by $\{|t|<t_d\}$, and $t_d\approx 0.594$. In the coordinates $(s,r)$, the positivity set is $$\Omega_U=\{(s,r): |s|< s_0 r \}\subset \mathcal{D},\quad\text{where}\quad s_0:=\tan\left(\frac13\arcsin t_d\right)\approx 0.215.$$
    Let $\tau\in(t_d,1)$ to be chosen later. In order to construct the first function $W_1$, we want to enlarge the positivity set of $U$ as follows. We fix 
    $$\overline s:=\tan\left(\frac13\arcsin \tau\right)>s_0\quad\text{and}\quad \overline\rho:=\sqrt{\overline s^2+1},$$ so that the enlarged cone is $|s|\le \overline s r$.
    Then, for $\alpha\in(-(d-2),0)$ to be chosen later, 
    we define for $r\ge 1$ $$W_1(x):=w_1(s,r)\quad\text{where}\quad w_1(s,r):=\begin{cases}
        c_d^{-1/2}\left(\rho f(t)+\lambda \rho^\alpha g_\alpha(t)\right)_+&\quad\text{for }|s|<\overline sr,\\
        0&\quad\text{otherwise}
    \end{cases},$$ where $g_\alpha$ is given by \eqref{eq.def:galp}, $\lambda:=-\overline\rho^{1-\alpha}\frac{f(\tau)}{g_\alpha(\tau)}$, and $(\rho,t)$ in terms of $(s,r)$ are given by $$\rho=\sqrt{s^2+r^2},\qquad t=\sin\left(3\arctan\left (\frac sr\right)\right).$$ 
    We notice that since $\tau>t_d$, the monotonicity of $f$ (see \eqref{eq:deriv}) implies that $f(\tau)<0$ and so, $\lambda>0$.
    Moreover, since \be\label{eq:dopo1}w_1(\overline s r,r)=c_d^{-1/2}\big(\overline\rho f(\tau)(r-r^\alpha)\big)_+=0\quad\text{for every }r\ge1,\ee the function $w_1$ is continuous.
    The function $h(s):=w_1(s,1)$ satisfies
    \begin{equation}\label{eq:def-h}
       h>0\quad\text{in }(-\overline s,\overline s)\qquad\text{and}\qquad h(\pm\overline s)=0. 
    \end{equation} 
    We also extend $h=0$ outside $[-\overline s,\overline s]$.
    For functions $y(r)$ and $v_2(r)$ to be chosen, we define for $r\in [r_\ast,1]$ $$W_2(x):=w_2(s,r)\quad\text{where}\quad w_2(s,r)= y(r)h\left(\frac{s}{v_2(r)}\right),$$
     and for $r\in(0,r_\ast]$ $$W_3(x):=w_3(s,r)\quad\text{and}\quad w_3(s,r)= y(r_\ast)h\left(\frac{s}{v_3(r)}\right),\quad\text{where}\quad v_3(r):=C_\ast r^{-3} $$
    and $C_\ast>0$ will be chosen in \eqref{eq:choiceCast}. 
    The function $W$ is then defined as in \eqref{eq:def-W}.

    \subsubsection{Compatibility conditions} Now we list the compatibility conditions, according to the bullet points above. 
    
    First, in order to prove that $W$ satisfies \cref{lemma:subsol-and-supersol}, we claim that it is sufficient to show that it is a strict supersolution, $B_{r_\ast}\subset \Omega_W$ and \be\tag{C1}\label{eq:cond-add}\overline s v_2(r)>s_0 r\quad\text{for every }r\in[r_\ast,1).\ee 
    Indeed, suppose that these three properties hold. Since $W>0$ in a neighborhood of the origin and $U(0)=0$, then $W>U$ in a neighborhood of the origin. Moreover, $W=W_1\ge U$ for $r\ge 1$. On the other hand, since the positivity set of $U$ is given by $|s|<s_0 r$ and the positivity set of $W$ is $|s|< \overline s v_2(r)$ for $r\in[r_\ast,1)$, then \eqref{eq:cond-add} ensures that $\Omega_U\subset \Omega_W$ for $r\in[r_\ast,1)$. Moreover, since for $r\in(0,r_\ast)$, $W$ is positive, the same inclusion holds also for $r\in(0,r_\ast)$.
    Then, $W\ge U$ in $\R^d$ by the maximum principle.
    Moreover, for every $x\in\R^d$ and $\rho$ sufficiently large, $W_\rho(x):=\rho^{-1}W(\rho x)=\rho^{-1}W_1(\rho x)\to U$ as $\rho\to+\infty$. 
    This concludes the proof of the claim.

    Second, the fact that $W_1$ is harmonic in its positivity set follows by \cref{lemma:generallemma}. Then, we need to show the free boundary condition. The same computation in \cref{subsec:sub} implies that the free boundary condition is verified if \begin{equation}\tag{C2}\label{eq:cond1-supersol}G(\alpha,t)< G(\alpha,t_d)\quad\text{for every }t\in(t_d,\tau],\end{equation} where $G(\alpha,t)$ is defined in \eqref{eq:Galophat}. Note that the free boundary of $w_1$ lies in the cone $|s|\le \overline sr$, hence we need to verify \eqref{eq:cond1-supersol} only for $t\le \tau$.
    
    Furthermore, we need that the slope of the level set $\{w_1(s,r)=0\}$ at $(\overline s,1)$ is greater than the slope of $|s|= \overline s r$. This fact and \eqref{eq:dopo1} guarantee that the free boundary of $W_1$ lies strictly inside the open cone $|s|<\overline s r$ for $r>1$. Then, if \begin{equation}\label{eq:def-A}
        A(\overline s):=-\frac{\partial_s w_1(\overline s,1)}{\partial_r w_1(\overline s,1)},
    \end{equation} we need the condition \begin{equation}\tag{C3}\label{eq:cond2-supersol} A(\overline s)>\frac1{\overline s}.
    \end{equation}
    Notice that $y(r)$ must be non-negative, since $W_2$ cannot be negative. We ask that \begin{equation}\tag{C4}\label{eq:cond3-supersol}y(r)>0,\qquad\text{for every }r\in[r_\ast,1) \end{equation}
    Additionally, since $W$ must be continuous, we need $w_1(s,1)=w_2(s,1)$, and so we require that \begin{equation}\tag{C5}\label{eq:cond4-supersol}v_2(1)=y(1)=1.\end{equation}
    Notice that the positivity set of a function of the form $h(s/v(r))$ is given by $$\left\{(s,r)\in\mathcal{D}:\ |s|<\overline s v(r)\right\}.$$
    Then, it is convenient to impose that \begin{equation}\tag{C6}\label{eq:cond6-supersol} \overline s v_3(r)>\frac{r}{\sqrt{3}}\quad \forall r\in(0,r_\ast),\qquad \overline s v_2(r)<\frac{r}{\sqrt{3}}\quad\forall r\in(r_\ast,1],\qquad \overline s v_3(r_\ast)=\overline s v_2(r_\ast)=\frac{r_\ast}{\sqrt{3}}.\end{equation} 
    Note that the third identity in \eqref{eq:cond6-supersol} implies that \begin{equation}\label{eq:choiceCast}
        C_\ast=\frac{r_\ast^4}{\sqrt{3}\overline s},
    \end{equation} and thus $v_3(r)=\frac{r_\ast^4}{\sqrt{3}\overline s}r^{-3}.$
    Moreover, \eqref{eq:cond6-supersol} ensures that the free boundary of $W_2$ is $|s|=\overline s v(r)$ for $r\in(r_\ast,1)$, while $W_3$ is positive for $r\in(0,r_\ast)$. 
    Indeed, since $ |s|\le r/\sqrt{3}$, this follows by \be\label{eq:prima}\left|\frac{s}{v_3(r)}\right|\le \frac{r/\sqrt{3}}{r_\ast^4/(\sqrt{3}\overline s)r^{-3}}\le \overline s\left(\frac{r}{r_\ast}\right)^4< \overline s\quad\text{for every }r\in(0,r_\ast),\ee and by \eqref{eq:def-h}.
    In particular, $B_{r_\ast}\subset\Omega_{W}$. Indeed, if $x\in B_{r_\ast}$, then $r\le |x|< r_\ast$, and thus $W_3(x)>0$. 
    
    Additionally, we ask that the two free boundaries of $W_1$ and $W_2$ have the same tangent, and thus \begin{equation}\tag{C7}\label{eq:cond5-supersol}v_2'(1)=\frac{1}{A(\overline s)\overline s}.\end{equation}
    Finally, we require the gluing of $W_2$ and $W_3$ to be $C^1$, and so we impose that \begin{equation}\tag{C8}\label{eq:cond7-supersol}  y'(r_\ast)=0,\qquad v_2(r_\ast)=\frac{r_\ast}{\sqrt{3}\overline s},\qquad v'_2(r_\ast)=-\frac{\sqrt 3}{\overline s}.\end{equation} Notice that this last condition also ensures that the free boundary $s=\pm \overline s v(r)$ meets the lines $s=\pm r/\sqrt 3$ orthogonally at $r=r_\ast$. Indeed, the free boundary ends in $r=\pm \sqrt 3 s$, whose tangent is $(1,\pm \sqrt 3)$. On the other hand, the free boundary $s=\pm \overline s v(r)$ has tangent $(\pm \overline s v'(r_\ast),1)$. Then we obtain the last condition above.
    Notice that this condition is the analogue of $v'(0)=0$ in the construction in \cite{DeSilvaJerison09:SingularConesIn7D}, which follows from the fact that $v$ is even in $r$.

    \subsubsection{The operator $\mathcal{L}$}
    We now compute the operator $\mathcal{L}$ for a function of the form $w(s,r)=y(r)h(\xi)$, where $\xi:=s/v(r)\in(-\overline s,\overline s)$. First, we compute the derivatives
    \begin{equation}\label{eq:deriv1}
         \partial_sw=y\frac{h'}{v},\qquad\partial_rw=y'{h}-yh's\frac{v'}{v^2},
    \end{equation}
    $$\partial_{ss}w=y\frac{h''}{v^2},\qquad \partial_{rr}w=y''h-2y'h's\frac{v'}{v^2}+yh''s^2\frac{v'^2}{v^4}-ys\left(\frac{v''}{v^2}-\frac{2(v')^2}{v^3}\right)h'.$$
    Then
    \begin{align*}\mathcal{L}w&=y\frac{h''}{v^2}+y''h-2y'h's\frac{v'}{v^2}+yh''s^2\frac{(v')^2}{v^4}-ys\left(\frac{v''}{v^2}-\frac{2(v')^2}{v^3}\right)h'\\&\qquad-\frac{2(d-2)s}{r^2-3s^2}y\frac{h'}{v}+\frac{(d-2)(r^2-s^2)}{r(r^2-3s^2)}\left(y'{h}-yh's\frac{v'}{v^2}\right)
    \\&=\left(1+s^2\frac{(v')^2}{v^2}\right)\frac{y}{v^2}h''\\&\qquad+\left(-2\frac{y'}{y}\frac{v'}{v}-\left(\frac{v''}{v}-\frac{2(v')^2}{v^2}\right)-\frac{2(d-2)}{r^2-3s^2}-\frac{(d-2)(r^2-s^2)}{r(r^2-3s^2)}\frac{v'}{v}\right)y\frac{s}{v}h'\\&\qquad\qquad+\left(y''+\frac{(d-2)(r^2-s^2)}{r(r^2-3s^2)}y'\right)h.
    \end{align*}
    Therefore, \begin{equation}\label{eq:mathcalL}
        \mathcal{L}w={y(r)} \Big(D(\xi,r)h''(\xi)-M(\xi,r)\xi h'(\xi)+N(\xi,r)h(\xi)\Big),
    \end{equation}
    where $$D(\xi,r):=\frac{1}{v(r)^2}+\frac{v'(r)^2}{v(r)^2}\xi^2,$$
    $$M(\xi,r):=2\frac{y'(r)}{y(r)}\frac{v'(r)}{v(r)}+\frac{v''(r)}{v(r)}-2\frac{v'(r)^2}{v(r)^2}+\frac{2(d-2)}{r^2-3v(r)^2\xi^2}+\frac{(d-2)(r^2-v(r)^2\xi^2)}{r(r^2-3v(r)^2\xi^2)}\frac{v'(r)}{v(r)},$$
    $$N(\xi,r):=\frac{y''(r)}{y(r)}+\frac{(d-2)(r^2-v(r)^2\xi^2)}{r(r^2-3v(r)^2\xi^2)}\frac{y'(r)}{y(r)}.$$
    \subsubsection{Equation of $W_3$}
    In order to prove that $w_3$ is a supersolution of the one-phase problem, we only need to show that $w_3$ is a supersolution for $\mathcal{L}$, since there is no free boundary for $r\in[0,r_\ast)$. Since $v'(r)/v(r)=-3/r$ and $v''(r)/v(r)=12/r^2$, then, using \eqref{eq:mathcalL} with $v=v_3$, we have
$$D(\xi,r)\ge0,\qquad M(\xi,r)=\frac{12}{r^2}-\frac{18}{r^2}-\frac{d-2}{r^2}=-\frac{d+4}{r^2}\le0,\qquad N(\xi,r)=0.$$
Moreover, as in \eqref{eq:prima}, $|\xi|< \overline s$ for every $r\in(0,r_\ast)$.
If we impose that 
    \begin{equation}\tag{C9}\label{eq:cond7.5-supersol} h''(\xi)\le0\quad\text{and}\quad \xi h'(\xi)\le0\quad\text{for every }\xi\in[-\overline s,\overline s],
    \end{equation} 
    then it follows that $$\mathcal{L}w_3\le0\quad\text{for every }r\in (0,r_\ast).$$
    We also observe that, using \eqref{eq:deriv1}, $$\partial_s w_3=y(r_\ast)\frac{h'(\xi)}{v_3(r)}\quad\text{and}\quad \partial_r w_3=y(r_\ast)h'(\xi)s\frac{3}{rv_3(r)}.$$ 
 Then, we have $$\partial_r w_3\mp\sqrt{3}\partial_s w_3=0\quad \text{on }s=\pm r/\sqrt{3},\quad\text{for every }r\in(0,r_\ast).$$ 
Notice that this is precisely the compatibility condition at the boundary of $\mathcal D$: it cancels the apparent singularity of the operator $\mathcal{L}$ along the lines $s=\pm r/\sqrt3$. It is the analogue of the condition $\partial_r w=0$ on the axis $r=0$ in the construction of \cite{DeSilvaJerison09:SingularConesIn7D}.
 
\subsubsection{Equation of $W_2$}
In order to prove that $w_2$ is a supersolution for $r\in (r_\ast,1)$, we choose $v_2(r)$ and $y(r)$.
By \eqref{eq:cond4-supersol}, \eqref{eq:cond5-supersol}, \eqref{eq:cond7-supersol}, we have four conditions on $v(r)$. Since we need a free parameter, we seek $v$ of the form $$v_2(r):=\mu+c_2r^2+c_4r^4+c_6r^6+c_8 r^8,$$ for some $\mu\in\R$ to be chosen.
Given $\mu$, we define $(c_2,c_4,c_6,c_8)$ as
$$\begin{pmatrix}
c_2 \\
c_4 \\
c_6 \\
c_8
\end{pmatrix}
:=
\begin{pmatrix}
1 & 1 & 1 & 1 \\
2 & 4 & 6 & 8 \\
r_\ast^2 & r_\ast^4 & r_\ast^6 & r_\ast^8 \\
2 r_\ast & 4 r_\ast^3 & 6 r_\ast^5 & 8 r_\ast^7
\end{pmatrix}^{-1}
\begin{pmatrix}
1 - \mu \\
1/(A(\overline{s})\overline{s}) \\
r_\ast/(\overline{s}\sqrt{3}) - \mu \\
-\sqrt{3}/\overline{s}
\end{pmatrix}.$$
Next we compute the free boundary condition. On the free boundary $|s|=\overline s v_2(r)$, we have $\xi=\pm\overline s$. In particular, $h(\xi)=0$. Then, using \eqref{eq:deriv1}, we have, for $v=v_2$
$$(\partial_sw_2)^2+(\partial_rw_2)^2=y(r)^2h'(\overline s)^2D(\overline s,r).$$ Then, in order to show the free boundary condition, we need to check that \begin{equation}\tag{C10}\label{eq:cond8.5-supersol}
    y(r)^2h'(\overline s)^2D(\overline s,r)<1\quad\text{for every }r\in[r_\ast,1].
\end{equation}
Additionally, we also need the condition  $$\partial_r w_1(s,1)\le\partial_r w_2(s,1).$$ This ensures that the gluing of $W_1$ and $W_2$ is superharmonic in its positivity set, once we know that $W_1$ and $W_2$ are superharmonic in their positivity sets.
Using again \eqref{eq:deriv1}, together with \eqref{eq:cond4-supersol} and \eqref{eq:cond5-supersol}, this last condition is equivalent to requiring that $$J(s):=\partial_r w_1(s,1)-\partial_r w_2(s,1)=\partial_r w_1(s,1)-y'(1)h(s)+\frac{h'(s)s}{A(\overline s)\overline s}\le0 \quad\forall s\in[0,\overline s].$$
By \eqref{eq:def-h}, \eqref{eq:def-A}, and since $h'(s)=\partial_s w_1(s,1)$, we have $J(\pm \overline s)=0$. Moreover, we also require that $J(0)=0$. Since $h'(0)=0$, this last condition is equivalent to
\begin{equation}\tag{C11}\label{eq:cond9-supersol}
    y'(1)=\frac{\partial_r w_1(0,1)}{h(0)}.
\end{equation}
Then, it is sufficient to show that \begin{equation}\tag{C12}\label{eq:cond10-supersol} J(s)\le0\quad\text{for every }s\in(0,\overline s).
\end{equation}
For $y(r)$, we have the three conditions \eqref{eq:cond4-supersol}, \eqref{eq:cond7-supersol} and \eqref{eq:cond9-supersol}. We seek $y(r)$ of the form 
$$y(r)=a_0+a_1r+a_2r^2,$$
 for some $a_0,a_1,a_2\in\R$.
Precisely, we set
$$\begin{pmatrix}
a_0 \\
a_1 \\
a_2 \\
\end{pmatrix}
:=
\begin{pmatrix}
1 & 1 & 1  \\
0 & 1 & 2r_\ast  \\
0 & 1 & 2  
\end{pmatrix}^{-1}
\begin{pmatrix}
1 \\
0 \\
\partial_r w_1(0,1)/h(0) 
\end{pmatrix}.$$

On the positivity set of $W_2$, namely where $h$ is positive, we have $\xi\in (-\overline s,\overline s)$. Then, the condition $\mathcal{L}w_2\le0$ for $r\in(r_\ast,1)$ is satisfied if we show that, for $v=v_2$ \be\tag{C13}\label{eq:cond11-supersol}F(\xi,r):=D(\xi,r)h''(\xi)-M(\xi,r)\xi h'(\xi)+N(\xi,r)h(\xi)\le0 \quad\text{for every }(\xi,r)\in \left(0,\overline s\right)\times (r_\ast,1),\ee 
where we restrict the interval of $\xi$ to $(0,\overline s)$, since $F$ is even in $\xi$.

\subsubsection{The choice of the parameters}
We choose $$\alpha:=-12, \qquad\tau:=0.62,\qquad r_\ast:=0.68,\qquad \mu:=10,$$
and we need to verify all the conditions listed
above.
\begin{itemize}
     \item \eqref{eq:cond4-supersol}, \eqref{eq:cond5-supersol}, \eqref{eq:cond7-supersol}, \eqref{eq:cond9-supersol} hold by definition of $v(r)$ and $y(r)$.
     \item \eqref{eq:cond-add} holds, by the figure below.
\begin{figure}[H]
  \centering
  \includegraphics[width=0.45\textwidth]{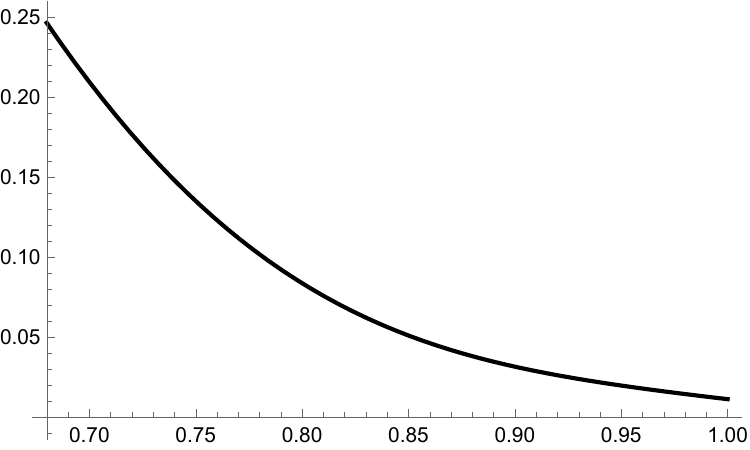}
  \caption{The graph of $\overline s v_2(r)-s_0r$ for $r\in [r_\ast,1)$.}
\end{figure}
    \item \eqref{eq:cond1-supersol} holds, by the figure below. \begin{figure}[H]
  \centering
  \includegraphics[width=0.45\textwidth]{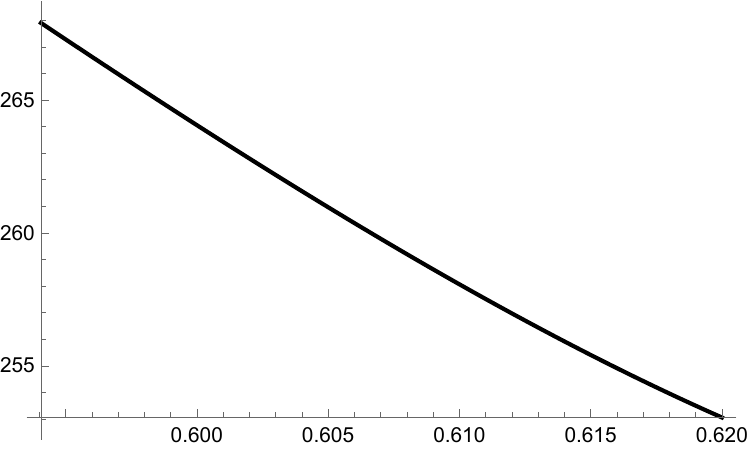}
  \caption{The graph of $G(\alpha,t)$ for $t\in [t_d,\tau]$.}
\end{figure}
\item \eqref{eq:cond2-supersol} holds, since $A(\overline s)\approx 16.26$ and $1/\overline s\approx 4.41$.
\item \eqref{eq:cond3-supersol} holds, by the figure below.
\begin{figure}[H]
  \centering
  \includegraphics[width=0.45\textwidth]{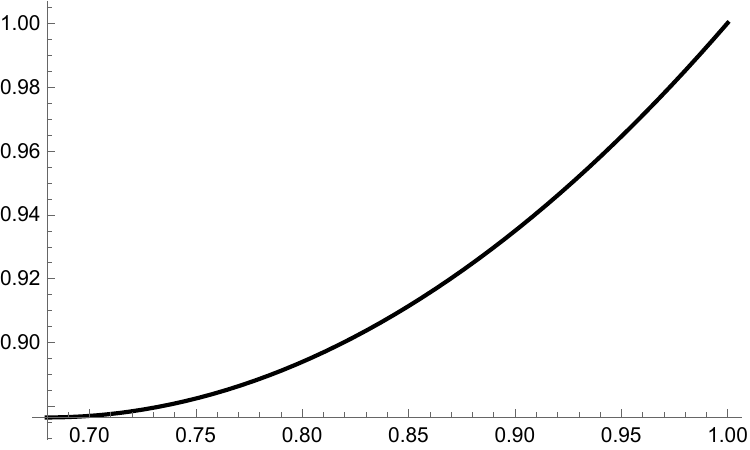}
  \caption{The graph of $y(r)$ for $r\in [r_\ast,1]$.}
\end{figure}

\item \eqref{eq:cond6-supersol} holds. Indeed, by definition of $v_3$, we have that $\overline sv_3(r)>\frac{r}{\sqrt{3}}$ for every $r\in (0,r_\ast)$. Moreover, the last condition of \eqref{eq:cond6-supersol} follows by definition of $v_3$ and \eqref{eq:cond7-supersol}. Finally, $\overline sv_2(r)<\frac{r}{\sqrt{3}}$ for every $r\in (r_\ast,1]$, by the figure below. 
\begin{figure}[H]
  \centering
  \includegraphics[width=0.45\textwidth]{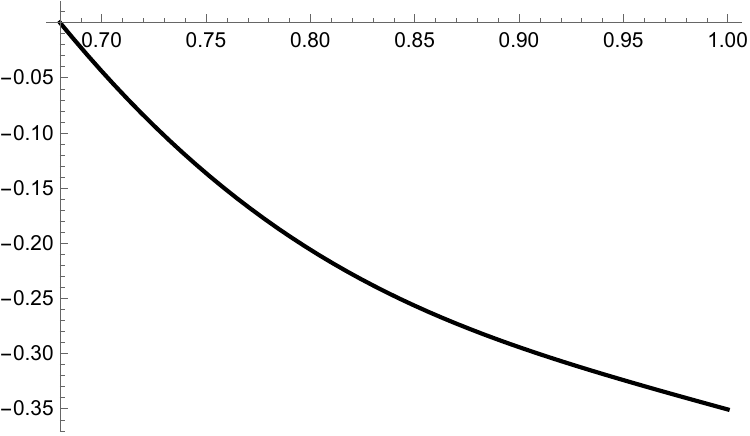}
  \caption{The graph of $\overline sv_2(r)-\frac{r}{\sqrt{3}}$ for $r\in [r_\ast,1]$.}
\end{figure}
\item \eqref{eq:cond7.5-supersol} holds, by the figures below.
\begin{figure}[H]
  \centering
  \includegraphics[width=0.45\textwidth]{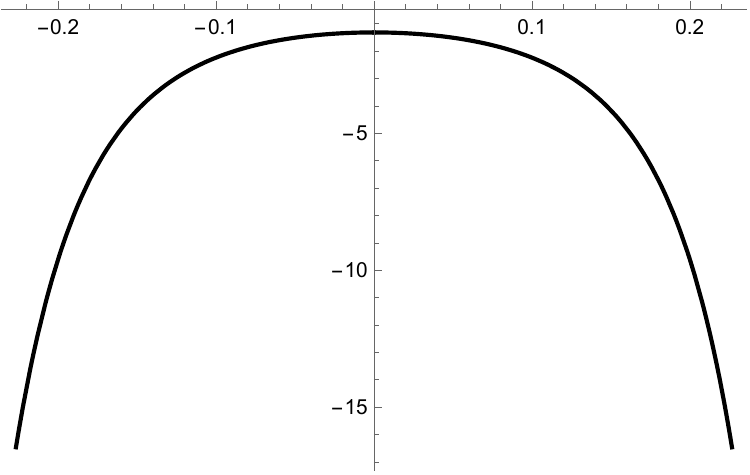}
  \caption{The graph of $h''(\xi)$ for $\xi\in [-\overline s,\overline s]$.}
\end{figure}
\begin{figure}[H]
  \centering
  \includegraphics[width=0.45\textwidth]{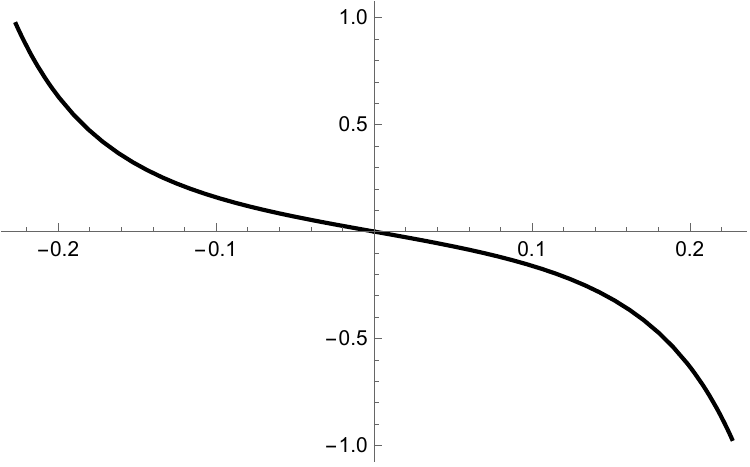}
  \caption{The graph of $h'(\xi)$ for $\xi\in [-\overline s,\overline s]$.}
\end{figure}

\item \eqref{eq:cond8.5-supersol} holds, by the figure below. Note that, numerically, if $B(r):=y(r)^2h'(\overline s)^2D(\overline s,r)$, we have $\sup_{r\in (r_\ast,1)}B(r)=B(r_\ast)\approx 0.964<1$.
\begin{figure}[H]
  \centering
  \includegraphics[width=0.45\textwidth]{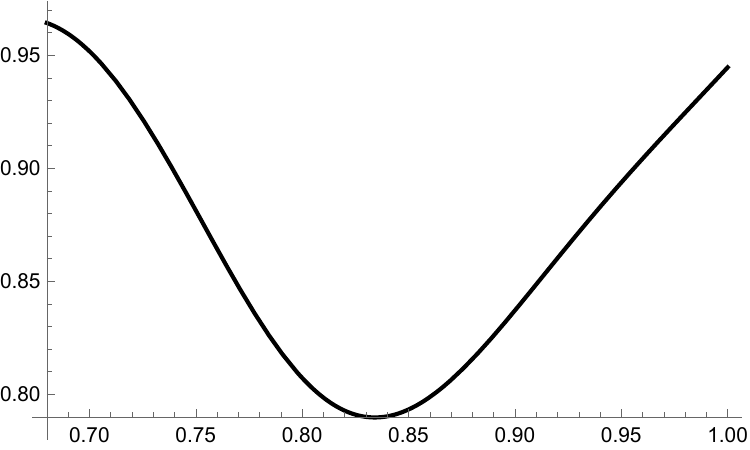}
  \caption{The graph of $B(r)$ for $r\in [r_\ast,1]$.}
\end{figure}

\item \eqref{eq:cond10-supersol} holds, by the figure below. 
   \begin{figure}[H]
  \centering
  \includegraphics[width=0.45\textwidth]{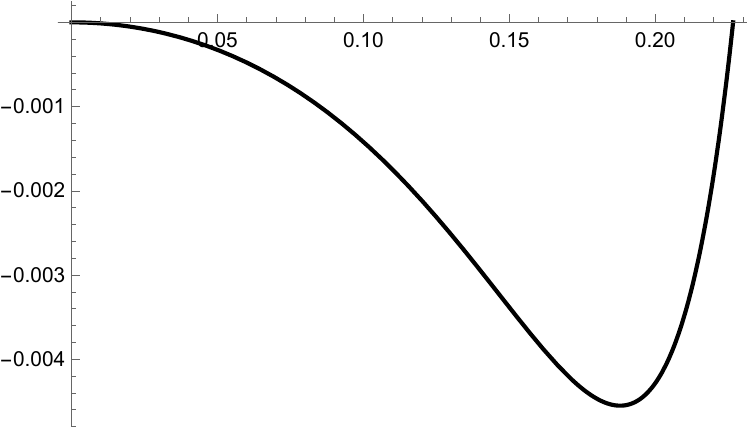}
  \caption{The graph of $J(s)$ for $s\in [0,\overline s]$.}
\end{figure}
\end{itemize} 
\subsubsection{Proof of $F(\xi,r)\le0$}
Finally, we need to prove \eqref{eq:cond11-supersol}. Since it is a two dimensional inequality, we simplify the proof using the following strategy. Let $m_1,m_2,m_3,m_4>0$ to be chosen later such that \be\label{eq:claim-bounds}\frac{M(\xi,r)}{D(\xi,r)}\le m_1+m_2\xi^2\quad\text{and}\quad\frac{N(\xi,r)}{D(\xi,r)}\le m_3+m_4\xi^2\quad\text{for every }(\xi,r)\in \left(0,\overline s\right)\times (r_\ast,1).\ee
Since $h(\xi)\ge0$, $\xi h'(\xi)\le 0$, then \eqref{eq:cond7.5-supersol} gives $$F(\xi,r)\le D(\xi,r)K(\xi),\quad \text{where}\quad K(\xi)=h''(\xi)-(m_1+m_2\xi^2)\xi h'(\xi)+(m_3+m_4\xi^2)h(\xi).$$ Therefore, in order to prove \eqref{eq:cond11-supersol}, we need to show \eqref{eq:claim-bounds} and \be\label{eq:claim2}K(\xi)\le 0\quad\text{for every }\xi\in(0,\overline s).\ee 
We notice that larger parameters $m_i$, with $i=1,2,3,4$, yield a simpler estimate \eqref{eq:claim-bounds}, but the function $K$ increases and could be positive, contradicting \eqref{eq:claim2}.
Then, the two inequalities \eqref{eq:claim-bounds} and \eqref{eq:claim2} are in competition, and we need to find a suitable choice of the parameters $m_i$ such that they hold at the same time.

Multiplying by $v(r)^2y(r)r Q(\xi,r)$, for $Q(\xi,r):=r^2-3v(r)^2\xi^2$, the inequalities in \eqref{eq:claim-bounds} are equivalent to 
\be\label{eq:eq1111}\begin{aligned}P_M(\xi,r)&:=\Big(2y'(r)v'(r)v(r)+v''(r)v(r)y(r)-2v'(r)^2y(r)\Big)rQ(\xi,r)\\&\qquad+2(d-2)v(r)^2y(r)r+(d-2)(r^2-v(r)^2\xi^2){v'(r)}v(r)y(r)\\&\qquad\qquad- (m_1+m_2\xi^2)y(r)r\left(1+v'(r)^2\xi^2\right)Q(\xi,r) \le0\end{aligned}\ee
\be\label{eq:eq1112}\begin{aligned}P_N(\xi,r)&:=\Big(r{y''(r)}Q(\xi,r)+(d-2)(r^2-v(r)^2\xi^2)y'(r)\Big)v(r)^2\\&\qquad- (m_3+m_4\xi^2)y(r)r\left(1+v'(r)^2\xi^2\right)Q(\xi,r)\le0.\end{aligned} \ee
For $\xi^2=\sigma$, these polynomials can be written as
$$P_M(\xi,r)=p_M(\sigma,r)=a_0^M(r)+a_1^M(r)\sigma+a_2^M(r)\sigma^2+a_3^M(r)\sigma^3,$$
$$P_N(\xi,r)=p_N(\sigma,r)=a_0^N(r)+a_1^N(r)\sigma+a_2^N(r)\sigma^2+a_3^N(r)\sigma^3.$$
for some $a_i^M(r)$ and $a_i^N(r)$, for $i=0,1,2,3$. More precisely, the coefficients corresponding to $i=2,3$ are
$$a^M_2(r):=r y(r)\big(3m_1v'(r)^2v(r)^2+3m_2v(r)^2-m_2r^2v'(r)^2\big),\quad a^M_3(r):=3m_2ry(r)v(r)^2v'(r)^2\ge0,$$
$$a^N_2(r):=r y(r)\big(3m_3v'(r)^2v(r)^2+3m_4v(r)^2-m_4r^2v'(r)^2\big),\quad a^N_3(r):=3m_4ry(r)v(r)^2v'(r)^2\ge0.$$
In particular, if we prove that $a^M_2(r)$ and $a^N_2(r)$ are non-negative on $(r_\ast,1)$, then the functions $\sigma\mapsto p_M(\sigma,r)$ and $\sigma\mapsto p_N(\sigma,r)$ are convex in $(0,\overline s^2)$ for every $r\in (r_\ast,1)$. Therefore, this should imply that
$$p_M(\sigma,r)\le (1-\sigma/\overline s^2)p_M(0,r)+\sigma/\overline s^2 p_M(\overline s^2,r),\quad p_N(\sigma,r)\le (1-\sigma/\overline s^2)p_N(0,r)+\sigma/\overline s^2 p_N(\overline s^2,r).$$
Thus, in order to prove \eqref{eq:eq1111} and \eqref{eq:eq1112}, we show that the two right-hand sides above are non-positive.
Recalling the definition of $p_M$ and $p_N$, then \eqref{eq:claim-bounds} holds if we choose parameters $m_i$ such that
\be\label{eq:eq111}a^M_2(r)\ge0,\quad P_M(0,r)\le0,\quad P_M(\overline s,r)\le0\quad\text{for every }r\in(r_\ast,1)\ee
\be\label{eq:eq112}a^N_2(r)\ge0,\quad P_N(0,r)\le0,\quad P_N(\overline s,r)\le0\quad\text{for every }r\in(r_\ast,1).\ee
To find suitable parameters $m_i$, we use the following strategy.
Given $m_2$ and $m_4$, we choose $m_1=m_1(m_2)$ and $m_3=m_3(m_4)$ by imposing \eqref{eq:eq1111} and \eqref{eq:eq1112}. Precisely, we take the supremum on the rectangle $(0,\overline s)\times (r_\ast,1)$ of a ratio of two polynomials. 
This allows to reformulate the problem in a choice of only two parameters $m_2$ and $m_4$.
We use the function NMaxValue in Mathematica, with DifferentialEvolution method. For $m_2=220$ and $m_4=50$, we find $m_1(m_2)\approx 62.55$ and $m_3(m_4)\approx 20.92$. Since it is convenient to increase the last two parameters slight, we choose the following parameters $$m_1:=63.2,\qquad m_2:=220,\qquad m_3:=21,\qquad m_4:=50.$$ 
This choice gives \eqref{eq:eq111}, \eqref{eq:eq112} and \eqref{eq:claim2}. The validity of all these estimates are shown in the following figures.
\begin{figure}[H]
  \centering
  \includegraphics[width=0.45\textwidth]{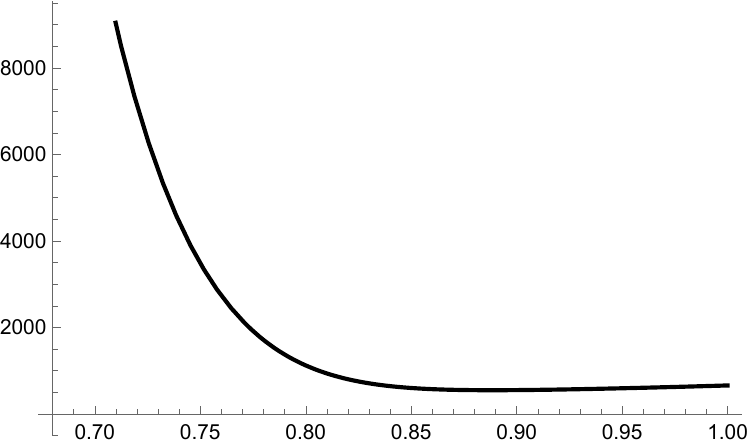}
  \caption{The graph of $a^M_2(r)$ for $r\in(r_\ast,1)$.}
\end{figure}
\begin{figure}[H]
  \centering
  \includegraphics[width=0.45\textwidth]{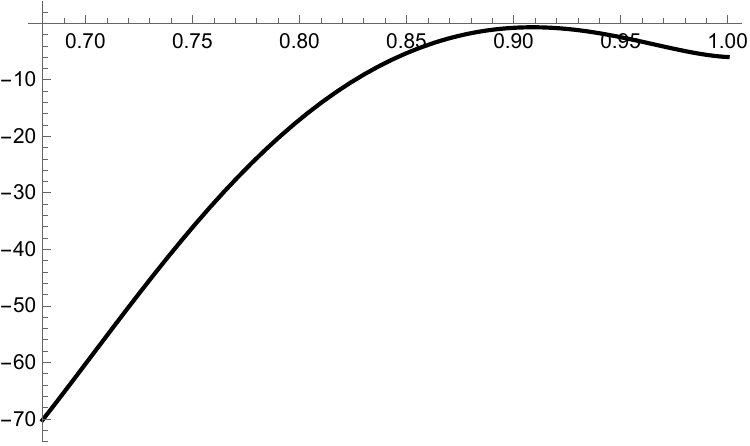}
  \caption{The graph of $P_M(0,r)$ for $r\in(r_\ast,1)$.}
\end{figure}
\begin{figure}[H]
  \centering
  \includegraphics[width=0.45\textwidth]{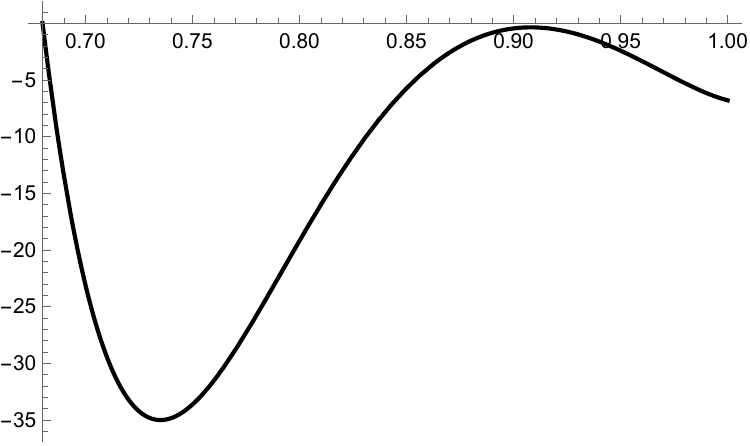}
  \caption{The graph of $P_M(\overline s,r)$ for $r\in(r_\ast,1)$.}
\end{figure}
\begin{figure}[H]
  \centering
  \includegraphics[width=0.45\textwidth]{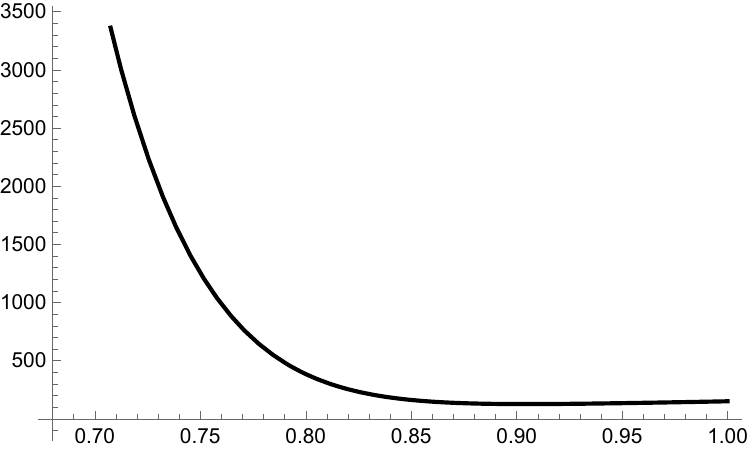}
  \caption{The graph of $a^N_2(r)$ for $r\in(r_\ast,1)$.}
\end{figure}
\begin{figure}[H]
  \centering
  \includegraphics[width=0.45\textwidth]{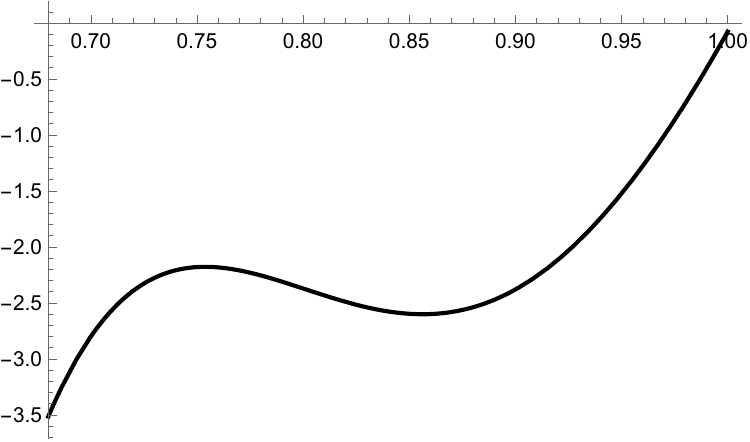}
  \caption{The graph of $P_N(0,r)$ for $r\in(r_\ast,1)$.}
\end{figure}
\begin{figure}[H]
  \centering
  \includegraphics[width=0.45\textwidth]{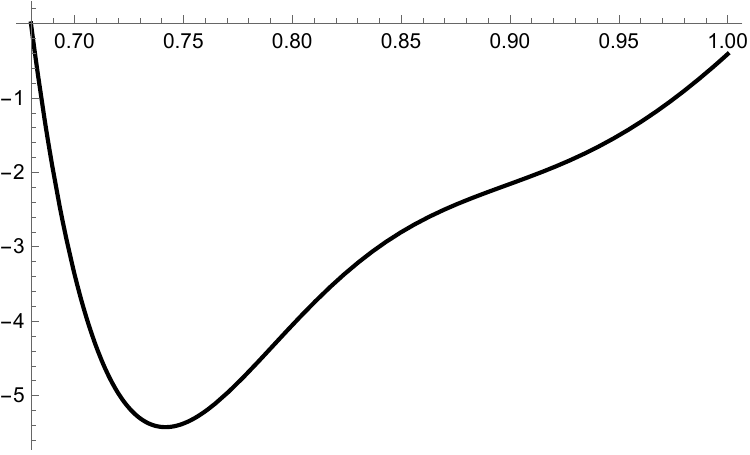}
  \caption{The graph of $P_N(\overline s,r)$ for $r\in(r_\ast,1)$.}
\end{figure}
\begin{figure}[H]
  \centering
  \includegraphics[width=0.45\textwidth]{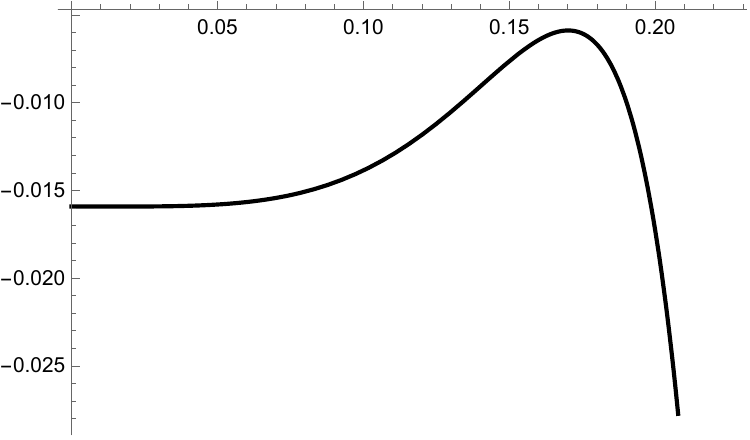}
  \caption{The graph of $K(\xi)$ for $\xi\in (0,\overline s)$.}
\end{figure}
\noindent More precisely, $$\inf_{r\in(r_\ast,1)}a^M_2(r)\approx 547.545>0,\ \sup_{r\in(r_\ast,1)}P_M(0,r)\approx -0.740<0 ,\ \sup_{r\in(r_\ast,1)}P_M(\overline s, r)=0,$$
$$\inf_{r\in(r_\ast,1)}a^N_2(r)\approx 126.925>0,\ \sup_{r\in(r_\ast,1)}P_N(0,r)\approx -0.077<0 ,\ \sup_{r\in(r_\ast,1)}P_N(\overline s, r)=0,$$
and $$\sup_{\xi\in (0,\overline s)}K(\xi)\approx -0.006<0.$$
This concludes the proof of \eqref{eq:claim-bounds} and \eqref{eq:claim2}, and thus the proof of \eqref{eq:cond11-supersol}.

\subsection{Conclusion of the proofs} Finally, we can conclude the proof of the minimality of $U_{26}$, and thus of our main result \cref{thm:main}.
\begin{proof}[Proof of \cref{prop:main2}]
    The result follows by \cref{lemma:subsol-and-supersol} and by the constructions of the subsolution and supersolution in \cref{subsec:sub} and \cref{subsec:super} respectively.
\end{proof}
\begin{proof}[Proof of \cref{thm:main}]
    The result follows by combining \cref{prop:main1} and \cref{prop:main2}.
\end{proof}

\appendix \section{Hypergeometric functions}

For $a,b,c\in\R$, with $c\not\in\{0,-1,-2,\ldots\}$, we recall that the hypergeometric function ${}_2F_1$ is defined as \be\label{def:hyper}{}_2F_1(a,b;c;s):=\sum_{k=0}^\infty \frac{(a)_k(b)_k}{(c)_k }\frac{s^k}{k!}\quad\text{for }s\in(-1,1),\ee where $(n)_k$ is the Pochhammer symbol defined as $(n)_0:=1$ and $(n)_k:=n(n+1)\cdots(n+k-1)$ for $k>0$. 
The derivative of the hypergeometric function ${}_2F_1$ is given by \be\label{eq:derivative-hyper}
\frac{d}{ds}{}_2F_1(a,b;c;s)=\frac{ab}{c}{}_2F_1(a+1,b+1;c+1;s).
\ee
Moreover, a straightforward term-by-term computation gives that $y(s):={}_2F_1(a,b;c;s)$ solves the ODE
\be\label{eq:odey}
    s(1-s)y''(s)+\left(c-(a+b+1)s\right)y'(s)-aby(s)=0\quad\text{in } (-1,1).
\ee

\begin{lemma}\label{lemma:hypergeometric-appendix}
    Let $A,B,C\in\R$ with $A\not=0$ and $(A+B)^2+4AC\ge0$. Then, the only solution of the ODE 
    \be\label{eq:sol-v}\begin{cases}
    A(1-t^2)h''(t)+Bth'(t)+Ch(t)=0\quad\text{in } (-1,1),\\
    h(0)=1,\  h'(0)=0,
\end{cases}\ee is given by the hypergeometric function
    \be\label{eq:def-h-geo}h(t)={}_2F_1\left(a,b;\frac12;t^2\right),\ee where \bea\label{eq:abc}a:=\frac{-(A+B)+\sqrt{(A+B)^2+4AC}}{4A},\quad b:=\frac{-(A+B)-\sqrt{(A+B)^2+4AC}}{4A}.\eea
\end{lemma}
\begin{proof}
Given $a,b,c\in\R$ to be chosen later, we set $y(s):={}_2F_1(a,b;c;s)$ and  $h(t):=y(t^2)$. Since $h(0)=1$ and $h'(0)=0$, then $h$ solves \eqref{eq:sol-v} if 
$$s(1-s)y''(s)+\left(\frac12+\frac{B-A}{2A}s\right)y'(s)+\frac{C}{4A}y(s)=0.$$
This is exactly \eqref{eq:odey} if we choose $a,b,c\in\R$ such that $$c=\frac12,\quad a+b=-\frac{A+B}{2A},\quad ab=-\frac{C}{4A}.$$
Writing $a,b$ in terms of $A,B,C$, we obtain that the function $h$ defined in \eqref{eq:def-h-geo} solves \eqref{eq:sol-v}, concluding the proof.
\end{proof}

\bibliography{biblio_complete}
\bibliographystyle{alpha}

\end{document}